\documentclass[11pt]{amsart}
\usepackage{amsmath,amscd}
\usepackage{amsbsy}
\usepackage{amssymb}
\usepackage{amscd,amsthm}

\usepackage[all,cmtip]{xy}

\newtheorem{thm}{Theorem}\numberwithin{thm}{section}
\newtheorem{lem}[thm]{Lemma}
\newtheorem{prop}[thm]{Proposition}
\newtheorem{cor}[thm]{Corollary}
\newtheorem{exam}[thm]{Example}
\newtheorem{rema}[thm]{Remark}
\newtheorem{con}[thm]{Conjecture}

\newtheorem{defi}[thm]{Definition}

\newtheorem*{thm2}{Theorem}
\newtheorem*{cor2}{Corollary}

\begin{document}
\begin{center}
\huge{Tilting objects on twisted forms of some relative flag varieties}\\[1,5cm]
\end{center}
\begin{center}

\large{Sa$\mathrm{\check{s}}$a Novakovi$\mathrm{\acute{c}}$}\\[0,5cm]
{\small March 18, 2015}\\[0,5cm]
\end{center}
{\small \textbf{Abstract}. 
We prove the existence of tilting objects on generalized Brauer--Severi varieties, some relative flags and some twisted forms of relative flags. As an application we obtain tilting objects on certain homogeneous varieties of classical type and on certain twisted forms of homogeneous varieties of type $A_n$ and $C_n$. 
\begin{center}
\tableofcontents
\end{center}
\section{Introduction}
The study of derived categories of coherent sheaves on schemes dates back to the 70's where Beilinson \cite{BE} described the bounded derived category of coherent sheaves on $\mathbb{P}^n$ and constructed tilting bundles. Later Kapranov \cite{KAP}, \cite{KAP1}, \cite{KAP2} constructed tilting bundles on certain homogeneous varieties and it was discovered that derived categories of coherent sheaves appeared in many areas of mathematics. Moreover, further examples of smooth projective schemes admitting a tilting bundle can be obtained from certain blow ups and taking projective bundles \cite{CMR1}, \cite{CRMR}, \cite{DO}. Note that a smooth projective scheme admitting a tilting bundle satisfies very strict conditions, namely its Grothendieck group is a free abelian group of finite rank and the Hodge diamond is concentrated on the diagonal (in characteristic zero) \cite{BH}. However, it is still an open problem to give a complete classification of smooth projective $k$-schemes admitting a tilting bundle. In the case of curves one can prove that a smooth projective algebraic curve has a tilting bundle if and only if the curve is a one-dimensional Brauer--Severi variety. But already for smooth projective algebraic surfaces there is currently no classification of surfaces admitting such a tilting object. It is conjectured that smooth projective algebraic surfaces have a tilting bundle if and only if they are rational (see \cite{BS}, \cite{HIP}, \cite{HIP2}, \cite{HIP3}, \cite{KI} and \cite{PE} for results in this direction). The results of Kapranov naturally led to the conjecture that homogeneous varieties should admit a full (strongly) exceptional collection. Up to now only partial results in favor of this conjecture are known (see \cite{SAM}, \cite{KUZ2}). In the present work we focus on relative versions of the classical results of Kapranov and on the problem of constructing tilting objects on certain twisted forms of homogeneous varieties. We start with a generalization of a result of Blunk \cite{BLU}. In loc.cit. the author proved the existence of tilting bundles on generalized Brauer--Severi varieties over fields of characteristic zero. We prove that tilting bundles exist on generalized Brauer--Severi varieties over arbitrary fields.
\begin{thm2}(Theorem 3.3)
Let $X=\mathrm{BS}(d,A)$ be a generalized Brauer--Severi variety over a field $k$ associated to a degree $n$ central simple $k$-algebra $A$. Then $X$ admits a tilting bundle.
\end{thm2} 
This theorem is proved by exploiting a result of Buchweitz, Leuschke and Van den Bergh \cite{BLB} stating a characteristic-free tilting bundle on the Grassmannian and by applying techniques from descent theory. The tilting bundle obtained by Blunk \cite{BLU} is constructed in this way that after base change the direct summands of that bundle become copies of the direct summands of Kapranov's tilting bundle. So Blunk's argument only works in characteristic zero because Kapranov's tilting bundle fails to be tilting in small characteristic (see \cite{BLB}). So for the proof of Theorem 3.3 one really needs a characteristic-free tilting bundle to apply descent theory. To prove some relative version of the above theorem, we first have to show that Grassmannian bundles admit tilting objects, generalizing in this way a result of Costa and Mir\'o-Roig \cite{CMR1} and Kapranov \cite{KAP}. For this, let $k$ be an algebraically closed field of characteristic zero, $X$ a smooth projective and integral $k$-scheme and $\mathcal{E}$ a locally free sheaf of finite rank on $X$. We then show the following: 
\begin{thm2}(Theorem 4.6)
Let $k$, $X$ and $\mathcal{E}$ be as above and suppose $X$ has a tilting bundle. Then the Grassmannian bundle $\mathrm{Grass}_X(l,\mathcal{E})$ admits a tilting bundle too. 
\end{thm2}
As the partial relative flag variety $\mathrm{Flag}_X(l_1,...,l_m,\mathcal{E})$ is obtained as the successive iteration of Grassmannian bundles, Theorem 4.6 provides us with a tilting object on $\mathrm{Flag}_X(l_1,...,l_m,\mathcal{E})$ (see Corollary 4.7). Now in order to get some relative version of Kapranov's result on smooth quadrics, we first prove a generalization of Theorem 4.6. It concerns smooth fibrations $\pi :X\rightarrow Z$ and combines ideas occurring in the work of Costa, Di Rocco and Mir\'o-Roig \cite{CRMR} and Samokhin \cite{SAM}. We fix the notation: Let $\pi:X\rightarrow Z$ be a flat proper morphism between two smooth projective schemes over an algebraically closed field of characteristic zero. Let $\mathcal{E}_1,...,\mathcal{E}_n$ be a set of locally free sheaves in $D^b(X)$ and suppose that for any point $z\in Z$ the restriction $\mathcal{E}^z_i=\mathcal{E}_i\otimes \mathcal{O}_{X_z}$ to the fiber $X_z$ is a full strongly exceptional collection for $D^b(X_z)$. Then one has:
\begin{thm2}(Theorem 4.15)
Let $\pi:X\rightarrow Z$ and $\mathcal{E}_i$ be as above and suppose that $D^b(Z)$ admits a tilting bundle $\mathcal{T}$. Then there exists an invertible sheaf $\mathcal{M}$ on $Z$ such that $\mathcal{R}=\bigoplus^n_{i=1}\pi^*(\mathcal{T}\otimes \mathcal{M}^{\otimes i})\otimes \mathcal{E}_i$ is a tilting bundle for $D^b(X)$.
\end{thm2}
Note that the more general approach of Theorem 4.15 also gives an alternative proof of Theorem 4.6. Now let $X$ be smooth projective and integral scheme over $\mathbb{C}$, $\mathcal{E}$ a locally free sheaf of finite rank and $q$ a symmetric quadratic form $q\in \Gamma(X,\mathrm{Sym}^2(\mathcal{E}^{\vee}))$ which is non-degenerate on each fiber. We denote by $\mathcal{Q}=\{q=0\}\subset \mathbb{P}(\mathcal{E})$ the quadric bundle and by $\pi$ the projection $\pi:\mathcal{Q}\rightarrow X$. Theorem 4.15 has the following consequence:
\begin{thm2}(Theorem 4.16)
Let $X$, $\mathcal{E}$ and $\mathcal{Q}$ be as above. Suppose $H^1(X,\mathbb{Z}/2\mathbb{Z})=0$ and that $\mathcal{E}$ is orthogonal and carries a spin structure. Suppose furthermore that $X$ admits a tilting bundle. Then $D^b(\mathcal{Q})$ admits a tilting bundle too.
\end{thm2}
As a consequence of Theorem 4.6 we can prove a generalization of Theorem 3.3 from above and a generalization of a result due to Yan \cite{YY}. Let $X$ be a smooth projective and integral $k$-scheme becoming rational after a separable field extension $k\subset L$. Consider a sheaf of Azumaya algebras $\mathcal{A}$ on $X$ and denote by $p:\mathrm{BS}(l_1,...,l_m,\mathcal{A})\rightarrow X$ a twisted form of a partial relative flag. 
\begin{thm2}(Corollary 5.5)
Let $X$ be as above, $k$ a field of characteristic zero and $p:\mathrm{BS}(l_1,...,l_m,\mathcal{A})\rightarrow X$ a twisted form of a partial relative flag. Suppose that $X$ becomes rational after a separable field extension $k\subset L$. If $X$ admits a tilting bundle, then $\mathrm{BS}(l_1,...,l_m,\mathcal{A})$ admits a tilting bundle too.
\end{thm2}
We believe that adapting the approach developed by Buchweitz, Leuschke and Van den Bergh \cite{BLB} in the relative setting would lead to a construction of a tilting bundle on $\mathrm{Grass}(l,\mathcal{E})$ over arbitrary fields $k$ and that therefore Corollary 5.5 holds without the assumption on $k$ being of characteristic zero. Samokhin \cite{SAM} proved that for $G$ a semisimple algebraic group of classical type and $B$ a Borel subgroup the flag variety $G/B$ admits a full exceptional collection. Collecting the above theorems, we observe that in this situation $G/B$ admits also a tilting object (see Theorem 5.1). The work of Panin \cite{PA} and Blunk \cite{BLU} provide some evidence to presume that twisted forms of homogeneous varieties should also posses a tilting object (see Conjecture 3.5 below). As an application of Corollary 5.5 we get that certain twisted forms of homogeneous varieties indeed admit tilting objects and provide further evidence for Conjecture 3.5. 
\begin{cor2}(Corollary 5.6, 5.7)
For a field $k$ of characteristic zero, let $X$ be a twisted form of the homogeneous variety $\mathrm{Sp}_{\bar{k}}(2n)/B$ specified in Section 5, or a twisted form of a partial flag $\mathrm{Flag}_{\bar{k}}(l_1,...,l_m,V)$, given as $\mathrm{BS}(l_1,...,l_m,A)$ where $A$ is a central simple algebra over $k$ obtained from $\mathrm{End}(V)$ by descent. Then $D^b(X)$ admits a tilting object.
\end{cor2}
To prove that certain twisted forms of homogeneous varieties of type $B_n$ and $D_n$ admit tilting bundles, one has to show that twisted forms of quadric fibrations have tilting bundles. In the absolute case this was done by Blunk \cite{BLU} by considering the involution variety of a central simple algebra with involution of the first kind. We believe that considering sheaves of Azumaya algebras with involutions of the first kind would lead to a similar result in the relative setting. The arguments presented in this work should then produce tilting bundles for twisted forms of homogeneous varieties of type $B_n$ and $D_n$.

The paper is organized as follows: In Section 2 we recall the basic facts about geometric tilting theory. Section 3 concerns generalized Brauer--Severi varieties and we prove that they allways admit a tilting bundle. In Section 4 we generalize Kapranov's classical results on tilting bundles on flags of type $A_n$ and quadrics to the relative setting and show that relative flag varieties of type $A_n$ and certain quadric fibrations admit tilting objects. In the last section we study twisted forms of relative flags, show that some of them admit tilting objects and apply the results to construct tilting bundles on twisted forms of some homogeneous varieties.\\

{\small \textbf{Acknowledgement}. This paper is based on a part of my Ph.D. thesis which was supervised by Stefan Schr\"oer whom I would like to thank for all his support and countless hours discussing the thesis with me. I also thank Markus Perling and Alexander Samokhin for inspiring conversations and helpful suggestions.\\

{\small \textbf{Conventions}. Throughout this work $k$ is an arbitrary field unless stated otherwise. 

\section{Generalities on geometric tilting theory}
In this section we recall some facts of geometric tilting theory. We start with the definition of a tilting object (see \cite{BH}).

\begin{defi}
Let $X$ be a noetherian quasiprojective $k$-scheme and $D(\mathrm{Qcoh}(X))$ the derived category of quasicoherent sheaves on $X$. An object $\mathcal{T}\in D(\mathrm{Qcoh}(X))$ is called \emph{tilting object for $D(\mathrm{Qcoh}(X))$} if the following hold:
\begin{itemize}
      \item[\bf (i)] $\mathrm{Hom}(\mathcal{T},\mathcal{T}[i])=0$ for $i\neq 0$.
      \item[\bf (ii)] If $\mathcal{N}\in D(\mathrm{Qcoh}(X))$ satisfies $\mathbb{R}\mathrm{Hom}(\mathcal{T},\mathcal{N})=0$, then $\mathcal{N}=0$.
			\item[\bf (iii)] $\mathrm{Hom}(\mathcal{T},-)$ commutes with direct sums.
\end{itemize}
\end{defi}
\begin{rema}
If one has a titling object $\mathcal{T}$ for $D(\mathrm{Qcoh}(X))$ one can form the smallest thick subcategory containing $\mathcal{T}$ that additionally is closed under direct sums. We denote this category by $\langle\mathcal{T}\rangle$. One can show that condition (ii) from above is equivalent to $\langle\mathcal{T}\rangle=D(\mathrm{Qcoh}(X))$ (see \cite{BH}, Remark 1.2). We say $\mathcal{T}$ is \emph{generating} the derived category $D(\mathrm{Qcoh}(X))$. Furthermore, if $D(\mathrm{Qcoh}(X))$ is compactly generated and the compact objects are exactly $D^b(X)$, then to show that an object $\mathcal{T}$ generates $D(\mathrm{Qcoh}(X))$ is equivalent to show that it generates $D^b(X)$, i.e., that the smallest thick subcategory containing $\mathcal{T}$ that additionally is closed under direct sums equals $D^b(X)$ (\cite{NEE}, see also \cite{BO3}, Theorem 2.1.2).
\end{rema}
Note that if $X$ is for instance a smooth projective and integral $k$-scheme, the derived category $D(\mathrm{Qcoh}(X))$ is compactly generated and the compact objects are exactly $D^b(X)$ (see \cite{NEE}, Theorem 2.5). In this case an object $\mathcal{T}$ generates $D(\mathrm{Qcoh}(X))$ if and only if it generates $D^b(X)$. Since the natural functor $D^b(X)\rightarrow D(\mathrm{Qcoh}(X))$ is fully faithfull (see \cite{HUY}), an object $\mathcal{T}$ lying in the subcategory $D^b(X)$ is a tilting object if and only if $\mathcal{T}$ generates $D^b(X)$ and $\mathrm{Hom}_{D^b(X)}(\mathcal{T},\mathcal{T}[i])=0$ for $i\neq 0$. If the tilting object $\mathcal{T}$ is a sheaf, the above definition coincides with the definition of a tilting sheaf given in \cite{B}. In this case the tilting object is called \emph{tilting sheaf} on $X$ or in $D^b(X)$. If it is a locally free sheaf we simply say that $\mathcal{T}$ is a \emph{tilting bundle}.\\
 
Now one has the following well-known tilting correspondence (see \cite{HVB}, Theorem 7.6).
\begin{thm}
Let $X$ be projective over a finitely generated $k$-algebra $R$. Suppose $X$ admits a tilting object $\mathcal{T}$ and set $A=\mathrm{End}(\mathcal{T})$. Then the following hold:
\begin{itemize}
      \item[\bf (i)] There is an equivalence $\mathbb{R}\mathrm{Hom}(\mathcal{T},-):D(\mathrm{Qcoh}(X))\stackrel{\sim}{\rightarrow} D(\mathrm{Mod}(A))$.
      \item[\bf (ii)] The equivalence of (i) restricts to an equivalence $D^b(X)\stackrel{\sim}{\rightarrow} D^b(A)$.
			\item[\bf (iii)] If $X$ is smooth over $R$, then $A$ has finite global dimension.
		
    \end{itemize}
\end{thm}
Assuming the existence of a tilting object $\mathcal{T}\in D(\mathrm{Qcoh}(X))$, where $X$ is a smooth projective and integral $k$-scheme, Theorem 2.3 gives an equivalence 
\begin{center}
$\mathbb{R}\mathrm{Hom}(\mathcal{T},-):D^b(X)\longrightarrow D^b(A)$,
\end{center} 
where $A=\mathrm{End}(\mathcal{T})$ is a finite dimensional $k$-algebra. If the field $k$ is supposed to be algebraically closed, any finite-dimensional $k$-algebra $A$ admits a \emph{complete set of primitive orthogonal idempotents} $e_1,...,e_n$ (see \cite{ASS}, I.4). Idempotents $e_1,...,e_n$ are called \emph{orthogonal} if $e_ie_j=e_je_i=0$ for $i\neq j$ and \emph{complete} if $e_1+...+e_n=1$. Furthermore, an idempotent $e$ is called \emph{primitive} if it cannot be written as a sum of two non-zero orthogonal idempotents. Now let $e_1,...,e_n$ be the complete set of primitive orthogonal idempotemts of the above endomorphism algebra $A$. Associated to $A$, there is a finite-dimensional $k$-algebra $A'$ with a complete set of primitive orthogonal idempotents $e'_1,...,e'_r$ such that $e'_iA'=e'_jA'$ as right $A$-modules only if $i=j$ (see \cite{ASS}, I.6, Definition 6.3). There is an equivalence of categories between $\mathrm{mod}(A)$ and $\mathrm{mod}(A')$ (see \cite{ASS}, I.6, Corollary 6.10). Now to every such algebra $A'$, with $e'_iA'=e'_jA'$ as right $A$-modules only if $i=j$, one can associate a quiver with relations $(Q,R)$ as follows: The set $Q_0$ is given by the set $e'_1,...,e'_r$ and the number of arrows from $e'_i$ to $e'_j$ is given by $\mathrm{Ext}^1_{A'}(S_i,S_j)$, where $S_l=e'_lA'/e'_l\mathrm{rad}(A')$ (see \cite{ASS}, II.3, Definition 3.1, see also \cite{ARS}, p.52 and Proposition 1.14 ). Note that this quiver does not depend on the choice of the set complete idempotents (see \cite{ASS}, II.3, Lemma 3.2). Moreover, the quiver $(Q,R)$ is uniquely determined up to isomorphism by $A'$ and the path algebra of $(Q,R)$ is isomorphic to $A'$ (see \cite{ARS}, III Theorem 1.9, Corollary 1.10). This yields an equivalence between $\mathrm{mod}(A)$ and $\mathrm{mod}(kQ/\langle R \rangle)$ and hence between $D^b(X)$ and $D^b(\mathrm{mod}(kQ/\langle R \rangle))$. Under this equivalence, each projective module $e'_iA'$ is mapped to a direct summand $\mathcal{E}_i$ of the tilting object $\mathcal{T}$ and the direct sum $\mathcal{T'}=\bigoplus^r_{i=1}\mathcal{E}_i$ is again a tilting object for $D^b(X)$. The difference between $\mathcal{T}$ and $\mathcal{T}'$ is that $\mathcal{T}$ may contain several copies of $\mathcal{E}_i$. This equivalence between $D^b(X)$ and $D^b(\mathrm{mod}(kQ/\langle R \rangle))$ now enables to apply representation-theoretical techniques to investigate the derived category of coherent sheaves on $X$. As a classical example we consider the tilting bundle $\mathcal{T}=\mathcal{O}_{\mathbb{P}^1}\oplus \mathcal{O}_{\mathbb{P}^1}(1)$ on the projective line $\mathbb{P}^1$. The corresponding quiver consists of two vertices and two arrows from the first vertex to the second $1\rightrightarrows 2$ and the representations of this quiver were studied by Kronecker and are well-known. For details and further examples of quivers related to tilting objects we refer to \cite{B}, \cite{BO}, \cite{BR}, \cite{BR1}, \cite{CS}, \cite{C}, \cite{ME} and \cite{PE1}.

Next we state some well-known facts concerning tilting objects. 
\begin{prop}
Let $X$ be a smooth projective and integral $k$-scheme and $\mathcal{T}=\bigoplus^n_{i=1}\mathcal{T}_i\in D^b(X)$ a tilting object. Then for all integers $r_i>0$ the object $\bigoplus^n_{i=1}\mathcal{T}_i^{\oplus r_i}$ is a tilting object too.
\end{prop}
\begin{prop}
Let $X$ be a smooth projective and integral $k$-scheme and $\mathcal{T}\in D^b(X)$ a tilting object. Then for all invertible sheaves $\mathcal{L}$ on $X$, the object $\mathcal{T}\otimes\mathcal{L}$ is a tilting object too.
\end{prop}
The next two results are folklore. The relative versions of these results are proved in \cite{BH}, Proposition 2.6, 2.9. We start with the proposition stating that the class of tilting objects is closed under taking products. Considering the projections $p:X\times_k Y\rightarrow X$ and $q:X\times_k Y\rightarrow Y$, we write for $p^*\mathcal{F}\otimes^L q^*\mathcal{G}$ simply $\mathcal{F}\boxtimes\mathcal{G}$. With this notation one gets the following:
\begin{prop}
Let $X$ and $Y$ be smooth projective and integral $k$-schemes and $\mathcal{T}_X\in D^b(X)$ and $\mathcal{T}_Y\in D^b(Y)$ tilting objects for $D(\mathrm{Qcoh}(X))$ and $D(\mathrm{Qcoh}(Y))$ respectively, then $\mathcal{T}=\mathcal{T}_X\boxtimes \mathcal{T}_Y$ is a tilting object for $D(\mathrm{Qcoh}(X\times Y))$.
\end{prop}
\begin{prop}
Let $X$ be a smooth projective $k$-scheme admitting a tilting object $\mathcal{T}\in D^b(X)$ and $k\subset L$ a separable field extension. Then $\mathcal{T}\otimes_k L$ is a tilting object too. 
\end{prop}
That somehow the converse of Proposition 2.7 holds is also well-known and a proof for this fact is given in \cite{NO1}. It is the following observation.
\begin{prop}
Let $X$ be a smooth, projective and integral $k$-scheme and $k\subset L$ an arbitrary field extension. Now given an object $\mathcal{R}\in D^b(X)$, suppose that $\mathcal{R}\otimes_k L$ is a tilting object on $X\otimes_k L$. Then $\mathcal{R}$ is a tilting object for $D(\mathrm{Qcoh}(X))$.
\end{prop}

In the literature (in view of the Krull--Schmidt decomposition), instead of the tilting object $\mathcal{T}$ one often studied the set $\mathcal{E}_1,...,\mathcal{E}_n$ of its indecomposable, pairwise non-isomorphic direct summands. There is a special case where all the summands form a so-called full strongly exceptional collection. Closely related to the notion of a full strongly exceptional collection is that of a semiorthogonal decomposition. We recall the definition of an exceptional collection and a semiorthogonal decomposition respectively. We follow the definition given in \cite{HUY} and refer to the work of Bondal and Orlov \cite{BO1} for further details. 
\begin{defi}
Let $X$ be a noetherian quasiprojective $k$-scheme. An object $\mathcal{E}\in D^b(X)$ is called \emph{exceptional} if $\mathrm{End}(\mathcal{E})=k$ and $\mathrm{Hom}(\mathcal{E},\mathcal{E}[l])=0$ for all $l\neq 0$. An \emph{exceptional collection} is a collection of exceptional objects $\mathcal{E}_1$,...,$\mathcal{E}_n$, satisfying
\begin{itemize}
      \item[\bf (i)] $\mathrm{Hom}(\mathcal{E}_i,\mathcal{E}_j[l])=k$ for $l=0$ and $i=j$,
      \item[\bf (ii)] $\mathrm{Hom}(\mathcal{E}_i,\mathcal{E}_j[l])=0$ for all $l\neq 0$ and $i=j$,
			\item[\bf (iii)] $\mathrm{Hom}(\mathcal{E}_i,\mathcal{E}_j[l])=0$ for all $l\in \mathbb{Z}$ if $i>j$ .
\end{itemize}
An exceptional collection is called \emph{full} if the collection generates $D^b(X)$, i.e., if the smallest thick subcategory containing $\mathcal{E}_1,...,\mathcal{E}_n$ that additionally is closed under direct sums equals $D^b(X)$. If in addition $\mathrm{Hom}(\mathcal{E}_i,\mathcal{E}_j[l])=0$ for all $i,j$ and $l\neq 0$ the collection is called \emph{strongly exceptional}.
\end{defi} 
As a generalization one has the notion of a semiorthogonal decomposition of $D^b(X)$ (see \cite{BO1} or \cite{HUY}).
\begin{defi}
Let $X$ be as above. A collection $\mathcal{D}_1,...,\mathcal{D}_r$ of full triangulated subcategories is called a \emph{semiorthogonal decomposition} for $D^b(X)$ if the following properties hold:

\begin{itemize}
      \item[\bf (i)] The inclusion $\mathcal{D}_i \subset D^b(X) $ has a right adjoint $p:D^b(X)\rightarrow \mathcal{D}_i$.
      \item[\bf (ii)] $\mathcal{D}_j\subset \mathcal{D}_i^{\perp}=\{B\in D^b(X)|\mathrm{Hom}(A,B)=0$, $\forall A\in\mathcal{D}_i\}$ for $i>j$.
			\item[\bf (iii)] The collection $\mathcal{D}_i$ generates $D^b(X)$, i.e., the smallest thick subcategory containing all $\mathcal{D}_i$ that additionally is closed under direct sums equals $D^b(X)$.
\end{itemize}
For a semiorthogonal decomposition of $D^b(X)$ we write $D^b(X)=\langle\mathcal{D}_1,...,\mathcal{D}_r\rangle$.		
\end{defi}
\begin{exam}
If we have a full exceptional collection $\mathcal{E}_1,...,\mathcal{E}_n$ in $D^b(X)$, then by Lemma 1.58 in \cite{HUY} we have that the inclusion $\langle \mathcal{E}_i \rangle\rightarrow D^b(X)$ has a right adjoint. Furthermore, condition (ii) is fulfilled for $\mathcal{D}_i=\langle \mathcal{E}_i\rangle$ and since the collection $\mathcal{E}_1,...,\mathcal{E}_n$ is full, the collection $\langle \mathcal{E}_1 \rangle,...,\langle\mathcal{E}_n\rangle$ generates $D^b(X)$. Hence a full exceptional collection $\mathcal{E}_1,...,\mathcal{E}_n$ give rise to a semiorthogonal decomposition $D^b(X)=\langle \langle \mathcal{E}_1 \rangle,...,\langle \mathcal{E}_n \rangle \rangle$. 
\end{exam} 
Clearly, the direct sum of the exceptional objects in a full strongly exceptional collection gives rise to a tilting object but the converse is not true in general. However, it is easy to verify that, after possibly reordering,  the pairwise non-isomorphic indecomposable direct summands of a tilting object form a full strongly exceptional collection, provided the summands are invertible. Exceptional collections and semiorthogonal decompositions were intensively studied and we know quite a lot of examples of schemes admitting full exceptional collections or semiorthogonal decompositions. For a comprehensive overview we refer to \cite{BO1}, \cite{KUZ4} and references therein.

The above described connection between representation theory and the derived category of coherent sheaves is not the only motivation to study derived categories with exceptional collections or tilting objects. Another motivation comes from Kontsevich's Homological Mirror Symmetry conjecture \cite{KO}, see also \cite{HUY}, 13.2. Moreover, a conjecture of Dubrovin \cite{D} states that the quantum cohomology of a smooth projective variety $X$ is generically semisimple if and only if there exists a full exceptional collection in $D^b(X)$ and the validity of this conjecture would also provide evidence for the Homological Mirror Symmetry conjecture. Although this conjecture turned out to be wrong in general, it seems that there is still a relationship between the existence of full exceptional collections and its quantum cohomology (see \cite{BA}). Motivated by the Mirror Symmetry, in the recent past full strongly exceptional collections have also been considered in physics in the context of string theory, concretely in studying so-called $D$-branes (see for instance \cite{A}). Particular interest in exceptional collections also comes from local Calabi--Yau varieties. Consider the total space $\pi:\mathbb{A}(\omega_X)\rightarrow X$ for the canonical bundle $\omega_X$. This is a local Calabi-Yau variety and it follows from results of Bridgeland \cite{BR} that a full strongly exceptional collection $\mathcal{E}_i$ on $X$ can be extended to a cyclic strongly exceptional collection if and only if the pullbacks $\pi^*\mathcal{E}_i$ give rise to a tilting object on the total space $\mathbb{A}(\omega_X)$. In this situation the endomorphism algebra of the titling object for $\mathbb{A}(\omega_X)$ gives an example of a noncommutative  resolution in the sense of \cite{VDB}.

Smooth projective and integral $k$-schemes admitting tilting objects include projective spaces \cite{BE}, flag varieties of type $A_n$ (over the complex numbers) \cite{KAP2}, Grassmannians over arbitrary fields \cite{BLB}, certain toric varieties \cite{CMR}, \cite{PRN}, rational complex surfaces \cite{HIP3} and certain iterated projective bundles and fibrations \cite{CMR1}, \cite{CRMR}. Furthermore, tilting bundles are presumed to exist on rational homogeneous varieties (see \cite{BOE}) and several results in this direction have been proved (see \cite{KUZ2} for details). For tilting bundles on some stacks we refer to \cite{NO1} and the references therein.

\section{Tilting objects for generalized Brauer--Severi varieties}
In this section we show how the \emph{indecomposable absolutely split locally free sheaves} on Brauer--Severi varieties, as introduced in \cite{NO}, Definition 4.1, very naturally give rise to a tilting object. Moreover, we show that the generalized Brauer--Severi varieties $\mathrm{BS}(d,A)$ also admit tilting objects, obtained from the tautological sheaf on $\mathrm{BS}(d,A)\otimes_k \bar{k}\simeq \mathrm{Grass}(d,n)$. Note that tilting bundles on Brauer--Severi varieties were previously constructed by Blunk \cite{BLU}, at least if the base field $k$ is of characteristic zero. We found our tilting object independently during the work on Brauer--Severi varieties and it turns out that our tilting object is characteristic-free and a direct summand of the tilting object given in \cite{BLU}. For details on Brauer--Severi varieties and central simple algebras we refer to \cite{ART}, \cite{GIS}.

Beilinson \cite{BE} showed that $\bigoplus^n_{i=0}\mathcal{O}_{\mathbb{P}^n}(i)$ is a tilting bundle on $\mathbb{P}^n$. Later Kapranov \cite{KAP2} investigated whether exceptional collections exist on homogenous varieties $X$. In loc.cit. he proved that homogenous varieties of type $A_n$ and smooth quadrics have full strongly exceptional collections and hence a tilting object. In view of this results and other evidence, it was conjectured that every rational homogeneous variety $G/P$ over an algebraically closed field of characteristic zero should posses a full exceptional collection and several results in this direction were obtained (see \cite{KUZ2}) Furthermore, it is conjectured by Catanese that $G/P$ should posses a full strongly exceptional poset and a bijection between the elements of the poset and set of Schubert varieties of $G/P$ such that the partial order of the poset is the order induced by the Bruhat--Chevalley order (see \cite{BOE}). Since Brauer--Severi and generalized Brauer--Severi varieties are twisted forms of projective spaces respectively Grassmannianns, in this section we therefore study certain simple examples of twisted forms of homogeneous varieties.\\

Recall, a \emph{Brauer--Severi variety} $X$ is a scheme of finite type over $k$ such that $X\otimes_k\bar{k}\simeq \mathbb{P}^n_{\bar{k}}$. Indeed, one can show that there is a finite Galois extension $k\subset L$ such that $X\otimes_k L\simeq \mathbb{P}^n_L$ (see \cite{GIS}, Corollary 5.1.5). A field extension $k\subset L$ such that $X\otimes_k L\simeq \mathbb{P}^n_L$ is called \emph{splitting field}. It is well-known that $n$-dimensional Brauer--Severi varieties correspond to degree $n+1$ central simple $k$-algebras via $H^1(k,\mathrm{PGL}_{n+1})$ (see \cite{GIS}, 5.2). According to the Wedderburn theorem a central simple algebra $A$ is isomorphic to a matrix algebra over a central division algebra $D$. This division algebra is unique up to isomorphism and the degree of $D$ is called the \emph{index} of $A$, denoted by $\mathrm{ind}(A)$. In \cite{NO} we classified all \emph{indecomposable absolutely split locally free sheaves} on Brauer--Severi varieties. \emph{Absolutely split} locally free sheaves on Brauer--Severi varieties are sheaves becoming isomorphic to the direct sum of invertible ones after base change to a splitting field. In particular, we proved that for a given Brauer--Severi variety $X$ over a field $k$ there exist indecomposable locally free sheaves $\mathcal{W}_i$ such that $\mathcal{W}_i\otimes_k L\simeq \mathcal{O}_{\mathbb{P}_L^n}(i)^{\oplus \mathrm{ind}(A^{\otimes i})}$, where $A$ is the central simple algebra corresponding to $X$ and $k\subset L$ an arbitrary splitting field (see \cite{NO}, Section 6). Note that the locally free sheaves $\mathcal{W}_i$ are unique up to isomorphism (\cite{NO}, Proposition 3.3) and the indecomposable absolutely split locally free sheaves on $X$. With this notation we now prove: 
\begin{thm}
Let $X$ be a $n$-dimensional Brauer--Severi variety over an arbitrary field $k$ and let $A$ be the corresponding central simple $k$-algebra. Let $\mathcal{W}_i$ be the locally free sheaves from above. Then the locally free sheaf $\mathcal{T}=\bigoplus^{n}_{i=0} \mathcal{W}_i$ is a tilting bundle for $D^b(X)$.
\end{thm}
\begin{proof}
Let $\pi:X\otimes_k L\rightarrow X$ be the projection, where $L$ is an arbitrary splitting field of $X$. As mentioned above, one has $\pi^*\mathcal{W}_i\simeq \mathcal{O}_{\mathbb{P}_L^n}(i)^{\mathrm{ind}(A^{\otimes i})}$. Clearly, $\mathrm{rk}(\mathcal{W}_i)=\mathrm{ind}(A^{\otimes i})>0$ and therefore Proposition 2.4 yields that $\pi^*\mathcal{T}\simeq \bigoplus^n_{i=0}\mathcal{O}_{\mathbb{P}_L^n}(i)^{\oplus \mathrm{rk}(\mathcal{W}_i)}$ is a tilting bundle for $D^b(\mathbb{P}_L^n)$. Now Proposition 2.8 shows that $\mathcal{T}=\bigoplus^{n}_{i=0} \mathcal{W}_i$ is a tilting bundle for $D^b(X)$. Note that since $X$ is smooth, the endomorphism algebra $\mathrm{End}(\mathcal{T})$ has finite global dimension according to Theorem 2.3 (iii). This completes the proof.
\end{proof}
\begin{rema}
The locally free sheaves $\mathcal{W}_i$ are by construction indecomposable (see \cite{NO}). Furthermore, they are up to isomorphism the only indecomposable locally free sheaves becoming isomorphic to copies of $\mathcal{O}_{\mathbb{P}_L^n}(i)$ after base change to some splitting field $L$. In view of this fact we have chosen the tilting bundle $\mathcal{T}=\bigoplus^{n}_{i=0} \mathcal{W}_i$ somehow optimal, in the sense that all direct summands are indecomposable and the rank of $\mathcal{T}$ is minimal with respect to the property that $\mathcal{T}\otimes_k \bar{k}$ is a tilting bundle for $D^b(\mathbb{P}^n)$, given as the direct sum of invertible sheaves on $\mathbb{P}^n$. Such kind of "optimal" choices among tilting objects were also given by Hille and Perling \cite{HIP3}, where they constructed tilting bundles for rational surfaces such that the ranks of the direct summands are minimal.
\end{rema} 

We now want to study the generalized Brauer--Severi varieties $\mathrm{BS}(d,A)$ associated to a central simple $k$-algebra $A$ of degree $n$ (see \cite{BL} for details on generalized Brauer--Severi varieties). Recall that for a central simple $k$-algebra $A$ of degree $n$ and $1\leq d\leq n$ the \emph{generalized Brauer--Severi} variety $\mathrm{BS}(d,A)$ corresponding to $A$ is the projective scheme of left ideals of $A$ of dimension $d\cdot n$. If the central simple algebra $A$ is isomorphic to $M_n(k)$ one has $\mathrm{BS}(d,A)\simeq \mathrm{Grass}(d,n)$. For any field extension $k\subset L$, $\mathrm{BS}(d,A)\otimes_k L\simeq \mathrm{BS}(d,A\otimes_k L)$ and hence $\mathrm{BS}(d,A)\otimes_k L\simeq \mathrm{Grass}_L(d,n)$ for any splitting field $L$ of $A$ (see \cite{BL}). To apply the same arguments as in Theorem 3.1 in the case of generalized Brauer--Severi varieties, one needs a characteristic-free tilting bundle for $\mathrm{BS}(d,A)\otimes_k \bar{k}\simeq \mathrm{Grass}(d,n)$. Kapranov \cite{KAP} investigated Grassmannians in characteristic zero and constructed a tilting bundle by making use of the Borel--Weil--Bott Theorem. More precise, he proved that on a Grassmannian $\mathrm{Grass}(d,n)$ over a field $k$ of characteristic zero the locally free sheaf $\mathcal{T}=\bigoplus_{\lambda} \Sigma^{\lambda}(\mathcal{S})$ is a tilting bundle and the direct summands form a full strongly exceptional collection. Here $\mathcal{S}$ is the tautological sheaf on $\mathrm{Grass}(d,n)$, $\Sigma$ the Schur functor and the direct sum is taken over all partitions $\lambda$ with at most $d$ rows and at most $n-d$ columns (see \cite{ABW}, \cite{FU} for details on partitions and the Schur functor). Now the main problem in arbitrary characteristic is that first, there is no Borel--Weil--Bott theorem, and second, that Kaneda \cite{KAN} showed that the above bundle $\mathcal{T}$ remains a tilting bundle as long as $\mathrm{char}(k)\geq n-1$. The bundle $\mathcal{T}$ fails to be a tilting bundle in very small characteristic as shown by Buchweitz, Leuschke and Van den Bergh \cite{BLB}, 3.3. To be more precise, they showed that in $\mathrm{char}(k)=2$ the bundle from above cannot be a tilting bundle on $\mathrm{Grass}(2,4)$, since $\mathrm{Ext}^1(\mathrm{Sym}^2(\mathcal{S}),\bigwedge^2(\mathcal{S}))\neq 0$ and hence $\mathcal{T}$ has non-trivial self extension. But instead of taking the above bundle $\mathcal{T}$, the authors proved that $\mathcal{T}'=\bigoplus_{\lambda\in P'} \bigwedge^{\lambda '}(\mathcal{S})$ is a characteristic-free tilting bundle (see \cite{BLB}, Theorem 1.3). Here $\lambda '$ is the conjugate partition of $\lambda$ and $\bigwedge^{\alpha}(\mathcal{S})=\bigwedge^{\alpha_1}(\mathcal{S})\otimes...\otimes\bigwedge^{\alpha_s}(\mathcal{S})$ for an arbitrary partition $\alpha=(\alpha_1,...,\alpha_s)$. As in characteristic zero, the sum is taken over all partitions $\lambda\in P'$. Here $P'$ is the set of all partitions $P(d,n-d)$ with at most $d$ rows and at most $n-d$ columns equipped with a total ordering $\prec$ such that if $|\lambda|<|\mu|$ then $\lambda\prec\mu$.\\

To guarantee that the same argument as in the proof of Theorem 3.1 works we have to consider the direct summands $\bigwedge^{\lambda '}(\mathcal{S})$ and investigate if they descent. Let $\lambda=(\lambda_1,...,\lambda_d)$, where $0\leq\lambda_i\leq n-d$, be a partition with at most $d$ rows and at most $n-d$ columns, then the conjugate partition $\lambda '=(\lambda '_1,...,\lambda '_{n-d})$ has at most $n-d$ rows and at most $d$ columns. We consider $\bigwedge^{\lambda '}(\mathcal{S})=\bigwedge^{\lambda '_1}(\mathcal{S})\otimes...\otimes\bigwedge^{\lambda '_{n-d}}(\mathcal{S})$ and investigate if $\bigwedge^{\lambda '_i}(\mathcal{S})$ descents. Note that $0\leq \lambda '_i\leq d$. By the definition of the Schur functor, $\bigwedge^{\lambda '_i}(\mathcal{S})=\Sigma^{\alpha_i}(\mathcal{S})$ , for the partition $\alpha_i=(1,1,...,1)$ with the ones in the partition occurring $\lambda '_i$-times. Since $0\leq \lambda'_i\leq d$, the partition $\alpha_i$ is a partition belonging to the set of partitions with at most $d$ rows and at most $n-d$ columns. One can show that the sheaves $\Sigma^{\alpha_i}(\mathcal{S})$ do not descent but $\Sigma^{\alpha_i}(\mathcal{S})^{\oplus n\cdot \lambda '_i}$ do (see \cite{LSW}, Section 4, p.114). Let $\mathcal{N}_{\lambda '_i}$ be the locally free sheaf such that $\mathcal{N}_{\lambda'_i}\otimes_k L\simeq \Sigma^{\alpha_i}(\mathcal{S})^{\oplus n\cdot \lambda '_i}$ for a Galois splitting field $k\subset L$ (note that $\mathrm{BS}(d,A)\otimes_k L\simeq \mathrm{Grass}(d,n)$ and $\mathcal{S}$ is the tautological sheaf on $\mathrm{Grass}(d,n)$). By \cite{NO}, Proposition 3.3, the locally free sheaf $\mathcal{N}_{\lambda'_i}$ is unique up to isomorphism. We set $\mathcal{N}_{\lambda '}=\mathcal{N}_{\lambda '_1}\otimes...\otimes \mathcal{N}_{\lambda '_{n-d}}$ and consider the locally free sheaf $\mathcal{T}=\bigoplus_{\lambda\in P'} \mathcal{N}_{\lambda '}$, where the sum is taken over all partitions $\lambda\in P'$.\\

With the above notation we get the following result:
\begin{thm}
Let $X=\mathrm{BS}(d,A)$ be a generalized Brauer--Severi variety over a field $k$ associated to the degree $n$ central simple $k$-algebra $A$. Then the locally free sheaf $\mathcal{T}=\bigoplus_{\lambda \in P'} \mathcal{N}_{\lambda '}$ from above is a tilting bundle for $D^b(X)$.
\end{thm}
\begin{proof}
Let $\pi:\mathrm{BS}(d,A)\otimes_k \bar{k}\rightarrow \mathrm{BS}(d,A)$ be the projection and $\lambda=(\lambda_1,...,\lambda_d)$ a partition with at most $d$ rows and $n-d$ columns. By the above discussion we have for $\lambda'=(\lambda'_1,...,\lambda'_{n-d})$, $\pi^*\mathcal{N}_{\lambda '_i}\simeq \Sigma^{\alpha_i}(\mathcal{S})^{\oplus n\cdot \lambda '_i}\simeq \bigwedge^{\lambda '_i}(\mathcal{S})^{\oplus n\cdot \lambda '_i}$, where $\mathcal{S}$ is the tautological sheaf on $\mathrm{BS}(d,A)\otimes_k \bar{k}\simeq \mathrm{Grass}(d,n)$. Hence
\begin{eqnarray*}
 \pi^*\mathcal{N}_{\lambda'}\simeq \pi^*\mathcal{N}_{\lambda '_1}\otimes...\otimes\pi^*\mathcal{N}_{\lambda '_{n-d}}\simeq \bigwedge^{\lambda '_1}(\mathcal{S})^{\oplus n\cdot \lambda '_1}\otimes...\otimes\bigwedge^{\lambda '_{n-d}}(\mathcal{S})^{\oplus n\cdot \lambda '_{n-d}}.
\end{eqnarray*} By Proposition 2.4 and \cite{BLB}, Theorem 1.3, the object 
\begin{eqnarray*}
\pi^*\mathcal{T}\simeq \bigoplus_{\lambda\in P'} \pi^*\mathcal{N}_{\lambda'}\simeq \bigoplus_{\lambda\in P'} (\bigwedge^{\lambda '_1}(\mathcal{S})^{\oplus n\cdot \lambda '_1}\otimes...\otimes\bigwedge^{\lambda '_{n-d}}(\mathcal{S})^{\oplus n\cdot \lambda '_{n-d}}) 
\end{eqnarray*}
is a characteristic-free tilting bundle for $D^b(\mathrm{BS}(d,A)\otimes_k\bar{k})$. Finally Proposition 2.8 shows that $\mathcal{T}$ is a tilting bundle for $D^b(\mathrm{BS}(d,A))$. Since $\mathrm{BS}(d,A)$ is smooth, the endomorphism algebra $\mathrm{End}(\mathcal{T})$ has finite global dimension according to Theorem 2.3. This completes the proof.
\end{proof}

\begin{cor}
Let $X=X_1\times...\times X_r$ be a finite product of generalized Brauer--Severi varieties, then $D^b(X)$ admits a tilting object. 
\end{cor}
In view of the fact that Brauer--Severi varieties in general do not have a full (strongly) exceptional collection (see \cite{NO1}, Proposition 4.6), for twisted forms of homogeneous varieties it is sensible to presume the existence of tilting objects rather then of a full (strongly) exceptional collection. The evidence coming from results of Panin \cite{PA} and Blunk \cite{BLU} naturally leads to the following conjecture.
\begin{con}
Let $X$ be a twisted form of a homogeneous variety, then $D^b(X)$ admits a tilting object.
\end{con} 
Note that in general one cannot expect to have a full strongly exceptional collection as in the case of homogeneous varieties over $\mathbb{C}$. Up to now, only partial results in direction of conjecture 3.5 are known. For instance Blunk \cite{BLU} constructed tilting bundles on twisted forms of quadrics and Theorem 3.3 from above shows that we have tilting bundles on twisted forms of Grassmannians. In Section 5 we also prove that the conjecture holds for certain twisted forms of homogeneous varieties of type $A_n$ and $C_n$.

\section{Tilting objects on some relative flag varieties}
In this section we prove the existence of tilting objects on some relative flags. More precisely, we prove that Grassmannian bundles and certain quadric bundles admit tilting objects. We start with a result that slightly generalizes a theorem proved by Costa and Mir\'o-Roig \cite{CMR}. In loc.cit., Theorem 4.1 below is proved in the case the base field $k$ is algebraically closed and of characteristic zero and $X$ is supposed to have a full strongly exceptional collection of invertible sheaves. We cite the next theorem as it is proved in \cite{NO1}, Theorem 5.3.
\begin{thm}
Let $k$ be a field, $X$ a smooth projective and integral $k$-scheme and $\mathcal{E}$ a locally free sheaf of finite rank. Suppose that $D^b(X)$ admits a tilting bundle, then $D^b(\mathbb{P}(\mathcal{E}))$ admits a tilting bundle too.
\end{thm}
\begin{exam}
Let $X$ be a generalized Brauer--Severi variety over $k$. According to Theorem 3.3, $X$ admits a tilting bundle. Theorem 4.1 applies and yields a tilting bundle for $D^b(\mathbb{P}(\mathcal{E}))$ for any finite rank locally free sheaf $\mathcal{E}$ on $X$. Notice that this example cannot be obtained from the result of  Costa and Mir\'o-Roig \cite{CMR} since first, $k$ is arbitrary and second, the tilting bundle on $X$ is not the direct sum of invertible sheaves.  
\end{exam}
To generalize Kapranov's result stating a tilting bundle for Grassmannians over $\mathbb{C}$ (see \cite{KAP}) , we consider the relative version of the Grassmannian, the Grassmannian bundle. For this, we take a smooth projective and integral $k$-scheme $X$ and a locally free sheaf $\mathcal{E}$ of rank $r+1$. Denote by $\mathrm{Grass}_X(l,\mathcal{E})$ the relative Grassmannian and by $\pi:\mathrm{Grass}_X(l,\mathcal{E})\rightarrow X$ the projective structure morphism. Furthermore, one has the tautological subbundle $\mathcal{R}$ of rank $l$ in $\pi^*\mathcal{E}$ and the tautological short exact sequence
\begin{eqnarray*}
0\longrightarrow \mathcal{R}\longrightarrow \pi^*\mathcal{E}\longrightarrow \mathcal{Q}\longrightarrow 0.
\end{eqnarray*}
\begin{rema}
As in the case of projective bundles, for an arbitrary invertible sheaf $\mathcal{L}$ on $X$ one has $\mathrm{Grass}_X(l,\mathcal{E})\simeq \mathrm{Grass}_X(l,\mathcal{E}\otimes \mathcal{L})$. 
\end{rema}
To see that the Grassmannian bundle admits a tilting sheaf, provided the base scheme $X$ admits one, we first state two lemmas. We fix some notations: Denote by $P(l, r+1-l)$ the set of partitions $\lambda=(\lambda_1,...,\lambda_l)$ with $0\leq \lambda_l\leq...\leq \lambda_1\leq r+1-l$. For $\lambda\in P(l, r+1-l)$ we have the Schur functor $\Sigma^{\lambda}$ and locally free sheaves $\Sigma^{\lambda}(\mathcal{R})$. Choose a total order $\prec$ on the set $P(l, r+1-l)$ such that for two partitions $\lambda$ and $\mu$, $\lambda\prec\mu$ means that the Young diagram of $\lambda$ is not contained in that of $\mu$, i.e., $\exists i: \mu_i<\lambda_i$. Let $P'$ be the above set of partitions equipped with this order. Suppose that $X$ is a smooth projective and integral $k$-scheme, where $k$ is an algebraically closed field of characteristic zero. Orlov \cite{DO} proved that $D^b(\mathrm{Grass}_X(l,\mathcal{E}))$ has a semiorthogonal decomposition 
\begin{eqnarray}
D^b(\mathrm{Grass}_X(l,\mathcal{E}))=\langle ..., D^b(X)\otimes \Sigma^\lambda(\mathcal{R}),...,D^b(X)\otimes \Sigma^\mu(\mathcal{R}),...\rangle,
\end{eqnarray} with $\lambda\prec \mu$ as defined above. Here $D^b(X)\otimes \Sigma^{\lambda}(\mathcal{R})$ is the full triangulated subcategory of $D^b(\mathrm{Grass}_X(l,\mathcal{E}))$ consisting of elements of the form $\pi^*\mathcal{M}\otimes \Sigma^{\lambda}(\mathcal{R})$, where $\mathcal{M}\in D^b(X)$. Furthermore, for all partitions $\lambda$ one has an equivalence between $D^b(X)$ and $D^b(X)\otimes \Sigma^{\lambda}(\mathcal{R})$, given by the functor $\pi^*(-)\otimes \Sigma^{\lambda}(\mathcal{R})$ (see \cite{DO}, \text{\S}3).
\begin{lem}
Let $k$ be an algebraically closed field of characteristic zero, $X$ a smooth projective and integral $k$-scheme, $P'$ the above ordered set of partitions and $\mathcal{E}$ a locally free sheaf of rank $r+1$ on $X$. Suppose the object $\mathcal{A}\in D^b(X)$ generates the category $D^b(X)$. Then the object $\mathcal{N}=\bigoplus_{\lambda\in P'}\pi^*\mathcal{A}\otimes \Sigma^{\lambda}(\mathcal{R})$ generates $D^b(\mathrm{Grass}_X(l,\mathcal{E}))$.
\end{lem}
\begin{proof}
In view of the equivalence $\pi^*(-)\otimes \Sigma^{\lambda}(\mathcal{R}):D^b(X)\stackrel{\sim}\rightarrow D^b(X)\otimes \Sigma^{\lambda}(\mathcal{R})$ and with the assumption that $D^b(X)$ is generated by the object $\mathcal{A}$, we conclude that $D^b(X)\otimes \Sigma^{\lambda}(\mathcal{R})$ is generated by the object $\pi^*\mathcal{A}\otimes \Sigma^{\lambda}(\mathcal{R})$. From the semiorthogonal decomposition (1) of $D^b(\mathrm{Grass}_X(l,\mathcal{E}))$ we finally get that the object $\mathcal{N}=\bigoplus_{\lambda\in P'}\pi^*\mathcal{A}\otimes \Sigma^{\lambda}(\mathcal{R})$ generates $D^b(\mathrm{Grass}_X(l,\mathcal{E}))$. 
\end{proof}
We now consider the higher direct images of $\Sigma^{\lambda}(\mathcal{R}^{\vee})$ under the Grassmannian bundle $\pi:\mathrm{Grass}(l,\mathcal{E})\rightarrow X$. We remind that we consider partitions $\lambda$ with at most $l$ rows and at most $r+1-l$ columns. One can extend the Schur functors and define them for a non-increasing sequence of arbitrary integers $\lambda_1\geq \lambda_2\geq...\geq \lambda_l$ by $\Sigma^{\lambda}(\mathcal{F}):=\Sigma^{\lambda+m}(\mathcal{F})\otimes \mathrm{det}(\mathcal{F})^{\otimes -m}$, where $\mathcal{F}$ is locally free of finite rank, $m\in \mathbb{N}$, $\lambda=(\lambda_1,...,\lambda_l)$ and $\lambda +m=(\lambda_1+m,...,\lambda_l+m)$. Note that $\Sigma^{\lambda}(\mathcal{F})^{\vee}\simeq \Sigma^{\lambda}(\mathcal{F}^{\vee})\simeq \Sigma^{-\lambda}(\mathcal{F})$, where $-\lambda=(-\lambda_l,-\lambda_{l-1},...,-\lambda_1)$. In the following we consider this extended Schur functors. The next lemma is well-known and can be found for instance in \cite{C1}, Lemma 4.1. 
\begin{lem}
Let $k$, $X$ and $\mathcal{E}$ be as above. Then for every partition $\lambda$, the higher direct images of $\Sigma^{\lambda}(\mathcal{R}^{\vee})$ under the Grassmann bundle $\pi:\mathrm{Grass}(l,\mathcal{E})\rightarrow X$ satisfy 
\[\mathbb{R}^s\pi_*(\Sigma^{\lambda}(\mathcal{R}^{\vee}))=\begin{cases}
\Sigma^{\lambda}(\mathcal{E}^{\vee})& \text{if s = 0 and if } \lambda_1\geq\lambda_2...\geq\lambda_l\geq 0\\
0& \text{otherwise} \\
\end{cases}\]
\end{lem}
\begin{proof}
In the case the base scheme $X$ is a point, this was done by Kapranov \cite{KAP}, Lemma 2.2 (a) (see also \cite{KAP2}, Lemma 3.2 a)). Precisely, Kapranov proved the following: Let $Z=\mathrm{Grass}(l,n)$ be the Grassmannian for some $n$-dimensional vector space $V$ over a field of characteristic zero and $\Sigma^{\lambda}(\mathcal{S}^{\vee})$ the locally free sheaf obtained by applying the Schur functor to the dual of the tautological sheaf $\mathcal{S}$ of $\mathrm{Grass}(l,n)$, where $\lambda$ is a non-increasing collection of integers $\lambda_1\geq \lambda_2\geq ...\geq \lambda_l\geq -(n-l)$. Then one has
\[H^s(Z,\Sigma^{\lambda}(\mathcal{S}^{\vee}))=\begin{cases}
\Sigma^{\lambda}(V^{\vee})& \text{if s = 0 and if } \lambda_1\geq\lambda_2...\geq\lambda_l\geq 0\\
0& \text{otherwise} \\
\end{cases}\]
Now for $x\in X$ we have $\pi^{-1}(x)\simeq \mathrm{Grass}(l,E_x)$, where $E_x$ is the fiber of $\mathcal{E}$ over $x$. The fiber of $\mathbb{R}^s\pi_*(\Sigma^{\lambda}(\mathcal{R}^{\vee}))$ over $x$ is $H^s(\mathrm{Grass}(l,E_x), \Sigma^{\lambda}(\mathcal{R}^{\vee})_{|\mathrm{Grass}(l,E_x)})$. By the definition of the tautological bundle $\mathcal{R}$, the restriction of $\mathcal{R}$ to $\mathrm{Grass}(l,E_x)$ is exactly the tautological rank $l$ bundle on $\mathrm{Grass}(l,E_x)$. We denote it by $S_x$. Hence the restriction of $\Sigma^{\lambda}(\mathcal{R}^{\vee})$ to the fiber over $x$ is $\Sigma^{\lambda}(S^{\vee}_x)$. The above result then follows from the result of Kapranov and by varying the point $x\in X$. This completes the proof.
\end{proof}
Note that Lemma 4.5 also follows from the more general Lemma 4.13 below. Now with the two lemmas from above we obtain the following result:
\begin{thm}
Let $k$, $X$ and $\mathcal{E}$ be as above and suppose $X$ has a tilting bundle $\mathcal{T}_X$. Then the Grassmannian bundle $\mathrm{Grass}_X(l,\mathcal{E})$ admits a tilting bundle too. 
\end{thm}
\begin{proof}
We prove that $\mathcal{T}=\bigoplus_{\lambda\in P'}\pi^*\mathcal{T}_X\otimes \Sigma^{\lambda}(\mathcal{R})$ is a tilting bundle on $\mathrm{Grass}_X(l,\mathcal{E})$. For this, we investigate when 
\begin{eqnarray}
\mathrm{Ext}^i(\bigoplus_{\lambda}\pi^*\mathcal{T}_X\otimes \Sigma^{\lambda}(\mathcal{R}),\bigoplus_{\lambda}\pi^*\mathcal{T}_X\otimes \Sigma^{\lambda}(\mathcal{R}))
\end{eqnarray} vanishes for $i>0$.
It is enough to calculate
\begin{center}
$\mathrm{Ext}^i(\pi^*\mathcal{T}_X\otimes \Sigma^{\lambda}(\mathcal{R}),\pi^*\mathcal{T}_X\otimes \Sigma^{\mu}(\mathcal{R}))$
\end{center} for two partitions $\lambda, \mu \in P'$. Adjunction of $\pi^*$ and $\pi_*$ and the projection formula yields
\begin{center}
$\mathrm{Hom}(\pi^*\mathcal{T}_X\otimes \Sigma^{\lambda}(\mathcal{R}),\pi^*\mathcal{T}_X\otimes \Sigma^{\mu}(\mathcal{R})[l])\simeq \mathrm{Hom}(\mathcal{T}_X,\mathcal{T}_X\otimes \mathbb{R}\pi_*(\Sigma^{\lambda}(\mathcal{R})^{\vee}\otimes \Sigma^{\mu}(\mathcal{R}))[l])$.
\end{center}
Hence we have to calculate $\mathbb{R}\pi_*(\Sigma^{\lambda}(\mathcal{R})^{\vee}\otimes \Sigma^{\mu}(\mathcal{R}))\simeq \mathbb{R}\pi_*\mathcal{H}om(\Sigma^{\lambda}(\mathcal{R}),\Sigma^{\mu}(\mathcal{R}))$. From the Littlewood--Richardson rule (see \cite{ABW}, Theorem IV.2.1), it follows that we can decompose $\mathcal{H}om(\Sigma^{\lambda}(-),\Sigma^{\mu}(-))$ into a direct sum of irreducible summands $\Sigma^{\gamma}(-)$. Since $\lambda$ and $\mu$ are partitions with at most $l$ rows and at most $r+1-l$ columns, it follows that $\gamma$ is a non-increasing sequence of integers $\gamma_1\geq \gamma_2\geq...\geq \gamma_l\geq -(r+1-l)$ (see \cite{KAP2}, 3.3). Now Lemma 4.5 yields for each irreducible summand $\Sigma^{\gamma}(\mathcal{R})\simeq \Sigma^{-\gamma}(\mathcal{R}^{\vee})$ of $\mathcal{H}om(\Sigma^{\lambda}(\mathcal{R}),\Sigma^{\mu}(\mathcal{R}))$ that
\begin{center}
$\mathbb{R}\pi_*(\Sigma^{\gamma}(\mathcal{R}))\simeq \mathbb{R}\pi_*(\Sigma^{-\gamma}(\mathcal{R}^{\vee}))\simeq \Sigma^{-\gamma}(\mathcal{E}^{\vee})$
\end{center} for $-\gamma\geq 0$, i.e., $-\gamma_l\geq -\gamma_{l-1}\geq...\geq -\gamma_1\geq 0$, otherwise $\mathbb{R}\pi_*(\Sigma^{\gamma}(\mathcal{R}))=0$. Since we have only finitely many partitions in $P'$ and finitely many irreducible summands $\Sigma^{\gamma}(\mathcal{R})$ of $\mathcal{H}om(\Sigma^{\lambda}(\mathcal{R}),\Sigma^{\mu}(\mathcal{R}))$, it is enough to prove the vanishing of 
\begin{center}
$\mathrm{Ext}^i(\mathcal{T}_X,\mathcal{T}_X\otimes \Sigma^{\gamma}(\mathcal{E}^{\vee}))$
\end{center} for $i>0$ and $\gamma \geq 0$. For the case $\gamma=0$ one has $\Sigma^{\gamma}(\mathcal{R}^{\vee})= \mathcal{O}_{\mathcal{E}}$ and hence $\mathbb{R}\pi_*\mathcal{O}_{\mathcal{E}}\simeq \mathcal{O}_X$. Therefore the vanishing of (2) in the case $\gamma=0$ follows from the fact that $\mathcal{T}_X$ is a tilting bundle on $X$ by assumption. It remains the case $\gamma >0$. Note that for an arbitrary locally free sheaf $\mathcal{F}$ and an arbitrary invertible sheaf $\mathcal{L}$ applying the Schur functor to $\mathcal{F}\otimes \mathcal{L}$ gives $\Sigma^{\gamma}(\mathcal{F}\otimes \mathcal{L})\simeq \Sigma^{\gamma}(\mathcal{F})\otimes \mathcal{L}^{\otimes |\gamma|}$, provided $\gamma\geq 0$. Since there are only finitely many summands $\Sigma^{\gamma}(\mathcal{R})$ of $\mathcal{H}om(\Sigma^{\lambda}(\mathcal{R}),\Sigma^{\mu}(\mathcal{R}))$ and $X$ is projective, we can choose for a fixed $\gamma > 0$ an ample invertible sheaf $\mathcal{L}$ and an integer $n_{\gamma}>>0$ such that 
\begin{center}
$\mathrm{Ext}^i(\mathcal{T}_X,\mathcal{T}_X\otimes \Sigma^{\gamma}((\mathcal{E}\otimes \mathcal{L}^{\otimes (-n_{\gamma})})^{\vee}))\simeq \mathrm{Ext}^i(\mathcal{T}_X,\mathcal{T}_X\otimes \Sigma^{\gamma}(\mathcal{E}^{\vee})\otimes \mathcal{L}^{\otimes (n_{\gamma}\cdot |\gamma|)})=0$
\end{center} for $i>0$. As mentioned above, there are only finitely many irreducible summands $\Sigma^{\gamma}(-)$ of $\mathcal{H}om(\Sigma^{\lambda}(-),\Sigma^{\mu}(-))$ so that we can consider an integer $n>\mathrm{max}\{n_{\gamma}| \mathrm{Ext}^i(\mathcal{T}_X,\mathcal{T}_X\otimes \Sigma^{\gamma}(\mathcal{E}^{\vee})\otimes \mathcal{L}^{\otimes (n_{\gamma}\cdot |\gamma|)})=0 $ for $i>0$$\}$. For this $n>>0$ we take the invertible sheaf $\mathcal{L}^{\otimes (-n)}$ and the Grassmannian bundle $\mathrm{Grass}(l,\mathcal{E}\otimes\mathcal{L}^{\otimes (-n)})$, denoting the respective tautological bundle on $\mathrm{Grass}(l,\mathcal{E}\otimes\mathcal{L}^{\otimes (-n)})$ by $\mathcal{R}'$. On this Grassmannian bundle we have for all $\gamma>0$ with $\mathbb{R}\pi_*(\Sigma^{\gamma}(\mathcal{R}'))=\Sigma^{\gamma}((\mathcal{E}\otimes \mathcal{L}^{\otimes (-n)})^{\vee})$:
\begin{center}
$\mathrm{Ext}^i(\mathcal{T}_X,\mathcal{T}_X\otimes \mathbb{R}\pi_*(\Sigma^{\gamma}(\mathcal{R})))\simeq \mathrm{Ext}^i(\mathcal{T}_X,\mathcal{T}_X\otimes \Sigma^{\gamma}(\mathcal{E}^{\vee})\otimes \mathcal{L}^{(n\cdot|\gamma|)}) =0$
\end{center} for $i>0$. This yields the vanishing of (2) for $\mathcal{T}'=\bigoplus_{\lambda\in P'}\pi^*\mathcal{T}_X\otimes \Sigma^{\lambda}(\mathcal{R}')$ on $\mathrm{Grass}(l,\mathcal{E}\otimes\mathcal{L}^{\otimes (-n)})$. Since $\mathcal{T}_X$ is a tilting bundle for $D^b(X)$, it generates $D^b(X)$ and according to Lemma 4.4 the object $\mathcal{T}'$ generates $D^b(\mathrm{Grass}(l,\mathcal{E}\otimes\mathcal{L}^{\otimes (-n)}))$. This gives us a tilting object $\mathcal{T}'$ on $\mathrm{Grass}(l,\mathcal{E}\otimes\mathcal{L}^{-n})$. By Remark 4.3 we have an isomorphism $\mathrm{Grass}(l,\mathcal{E}\otimes\mathcal{L}^{\otimes (-n)})\simeq \mathrm{Grass}(l,\mathcal{E})$ and hence we get a tilting object $\widetilde{\mathcal{T}}$ for $D^b(\mathrm{Grass}(l,\mathcal{E}))$. Finally, since $\mathrm{Grass}(l,\mathcal{E})$ is by assumption smooth over $k$, Theorem 2.3 implies that $\mathrm{End}(\widetilde{\mathcal{T}})$ has finite global dimension. This completes the proof.  
\end{proof}
Let $\mathcal{E}$ be a locally free sheaf of rank $r+1$ on some smooth projective and integral $k$-scheme $X$ as above. For $1\leq l_1<...<l_m\leq r+1$ consider the relative flag variety $\mathrm{Flag}_X(l_1,...,l_m,\mathcal{E})$ of type $(l_1,...,l_m)$ in the fibers of $\mathcal{E}$ with structure morphism $\pi:\mathrm{Flag}_X(l_1,...,l_m,\mathcal{E})\rightarrow X$. One has the tautological subbundles $\mathcal{R}_1\subset\mathcal{R}_2\subset ... \subset\mathcal{R}_m\subset \pi^*\mathcal{E}$ and, by construction, $\mathrm{Flag}_X(l_1,...,l_m,\mathcal{E})$ is obtained as the successive iteration of Grassmannian bundles
\begin{center}
$\mathrm{Flag}_X(l_1,...,l_m,\mathcal{E})=\mathrm{Grass}_{\mathrm{Flag}_X(l_2,...,l_m,\mathcal{E})}(l_1,\mathcal{R}_2)\longrightarrow \mathrm{Flag}_X(l_2,...,l_m,\mathcal{E})=\mathrm{Grass}_{\mathrm{Flag}_X(l_3,...,l_m,\mathcal{E})}(l_2,\mathcal{R}_3)\longrightarrow ...\longrightarrow X$.
\end{center} Theorem 4.6 now implies:
\begin{cor}
Let $k$ be an algebraically closed field of characteristic zero, $X$ a smooth projective $k$-scheme and $\mathcal{E}$ a locally free sheaf of finite rank on $X$. Suppose $\mathcal{T}$ is a tilting bundle for $X$. Then the relative flag $\mathrm{Flag}_X(l_1,...,l_m,\mathcal{E})$ admits a tilting bundle too.
\end{cor}

\begin{exam}
Let $S$ be a rational surface over $k=\mathbb{C}$ and $\mathcal{E}$ a locally free sheaf of finite rank on $S$. Hille and Perling \cite{HIP3} proved that $S$ always admits a tilting bundle $\mathcal{T}_S$. Now Corollary 4.7 implies that $\mathrm{Flag}_S(l_1,...,l_t,\mathcal{E})$ admits a tilting bundle too.
\end{exam}
B\"ohning \cite{BOE} constructed a semiorthogonal decomposition for quadric bundles under certain assumptions. In some following work Kuznetsov \cite{KUZ} studied quadric fibrations and intersections of quadrics and proved the existence of semiorthogonal decompositions in hole generality. We follow the ideas of \cite{BOE} and \cite{KUZ} to construct a tilting object on quadric bundles (quadric fibrations). We start with analyzing the result given in \cite{BOE}. For this, we take a smooth projective and integral scheme over $\mathbb{C}$. Let $\mathcal{E}$ be a locally free sheaf of rank $r+1$ and $q$ a symmetric quadratic form $q\in \Gamma(X,\mathrm{Sym}^2(\mathcal{E}^{\vee}))$ which is non-degenerate on each fiber. We denote by $\mathcal{Q}=\{q=0\}\subset \mathbb{P}(\mathcal{E})$ the quadric bundle and by $\pi$ the projection $\pi:\mathcal{Q}\rightarrow X$. Under some technical assumptions stated below, B\"ohning \cite{BOE} established two ordered sets of locally free sheaves on $\mathcal{Q}$
\begin{eqnarray}
\mathcal{V}=\{\Sigma(-r+1)\prec\mathcal{O}_{\mathcal{Q}}(-r+2)\prec...\prec\mathcal{O}_{\mathcal{Q}}(-1)\prec\mathcal{O}_{\mathcal{Q}}\} \\
\mathcal{V}'=\{\Sigma^+(-r+1)\prec\Sigma^-(-r+1)\prec...\prec\mathcal{O}_{\mathcal{Q}}(-1)\prec\mathcal{O}_{\mathcal{Q}}\}.
\end{eqnarray} 
We refer to \cite{BOE} for details on the twisted spinor bundles $\Sigma(-r+1)$, $\Sigma^+(-r+1)$ and $\Sigma^-(-r+1)$ in this relative setting. In loc.cit. it is proved the following (\cite {BOE}, Theorem 3.2.7): 
\begin{thm}
Let $X$ be as above, $\mathcal{E}$ an orthogonal locally free sheaf of rank $r+1$ on $X$ and $\mathcal{Q}$ the quadric bundle. Suppose $H^1(X,\mathbb{Z}/2\mathbb{Z})=0$ and that $\mathcal{E}$ carries a spin structure. Then there is a semiorthogonal decomposition 
\begin{eqnarray*}
D^b(\mathcal{Q})=\langle D^b(X)\otimes \Sigma (-r+1), D^b(X)\otimes \mathcal{O}_{\mathcal{Q}}(-r+2),\\
...,D^b(X)\otimes \mathcal{O}_{\mathcal{Q}}(-1), D^b(X)\rangle
\end{eqnarray*} for $r+1$ odd and 
\begin{eqnarray*}
D^b(\mathcal{Q})=\langle D^b(X)\otimes \Sigma^+ (-r+1), D^b(X)\otimes \Sigma^- (-r+1),\\
...,D^b(X)\otimes \mathcal{O}_{\mathcal{Q}}(-1), D^b(X)\rangle
\end{eqnarray*} for $r+1$ even.
\end{thm}
Note that the spin structure of $\mathcal{E}$ somehow guarantees the existence of the spinor bundles in the relative setting. The proof of Theorem 4.9 needs the following, for our purposes also, very important observation (see \cite{BOE}, Lemma 3.2.5).
\begin{lem}
Consider the two ordered sets (3) and (4) from above. If $\mathcal{W},\mathcal{V}_1,\mathcal{V}_2\in \mathcal{V}$ (resp. $\in \mathcal{V}'$) with $\mathcal{V}_1\prec \mathcal{V}_2$, $\mathcal{V}_1\neq \mathcal{V}_2$, then one has
\begin{itemize}
      \item[\bf (i)] $\mathbb{R}^i\pi_*(\mathcal{W}\otimes \mathcal{W}^{\vee})=0$, $\forall i\neq 0$
      \item[\bf (ii)] $\mathbb{R}^i\pi_*(\mathcal{V}_1\otimes \mathcal{V}^{\vee}_2)=0$, $\forall i\in\mathbb{Z}$
			\item[\bf (iii)] $\mathbb{R}^i\pi_*(\mathcal{V}_2\otimes \mathcal{V}^{\vee}_1)=0$, $\forall i\neq 0$
		\end{itemize} and the canonical morphism $\pi_*(\mathcal{W}\otimes \mathcal{W}^{\vee})\rightarrow \mathcal{O}_X$ is an isomorphism.
\end{lem}

\begin{lem}
Let $X$ and $\mathcal{E}$ be as above with all the assumptions on $X$ and $\mathcal{E}$ of Theorem 4.9 being fulfilled. Suppose the object $\mathcal{A}$ generates the category $D^b(X)$. Then the object 
\begin{eqnarray*}
\mathcal{N}=\bigoplus^r_{i=0}\pi^*\mathcal{A}\otimes \mathcal{V}_i,
\end{eqnarray*} where $\mathcal{V}_i$ are the elements of the set (3), generates $D^b(\mathcal{Q})$ for $r+1$ odd and the object
\begin{eqnarray*}
\mathcal{N}'=\bigoplus^{r+1}_{i=0}\pi^*\mathcal{A}\otimes \mathcal{V}'_i,
\end{eqnarray*} where $\mathcal{V}'_i$ are the elements of the set (4), generates $D^b(\mathcal{Q})$ for $r+1$ even.
\end{lem}

With Lemma 4.10 and 4.11 we now obtain the following:
\begin{prop}
Let $X$ be as above, $\mathcal{E}$ an orthogonal locally free sheaf of rank $r+1$ on $X$ and $\mathcal{Q}$ a smooth quadric bundle. Suppose $H^1(X,\mathbb{Z}/2\mathbb{Z})=0$ and that $\mathcal{E}$ carries a spin structure. Suppose furthermore that $\mathcal{T}_X$ is a tilting bundle for $D^b(X)$ and that $\mathrm{Ext}^l(\mathcal{T}_X,\mathcal{T}_X\otimes \pi_*(\mathcal{V}_j\otimes \mathcal{V}^{\vee}_i))=0$ for $l\neq 0$ and $\mathcal{V}_i\prec \mathcal{V}_j$, $\mathcal{V}_i\neq \mathcal{V}_j$, where $\mathcal{V}_i,\mathcal{V}_j\in \mathcal{V}$ (resp. $\in\mathcal{V}'$). Then $\mathcal{Q}\subset \mathbb{P}(\mathcal{E})$ admits a tilting bundle.
\end{prop}
\begin{proof}
We prove that $\mathcal{T}=\bigoplus_i\pi^*\mathcal{T}_X\otimes \mathcal{V}_i$, with $\mathcal{V}_i$ being elements of the set (3), is a tilting bundle in the odd case and $\mathcal{T}=\bigoplus_i\pi^*\mathcal{T}_X\otimes \mathcal{V}'_i$, with $\mathcal{V}'_i$ being elements of the set (4), in the even case. We give the proof only for the odd case and note that the proof for the even case is completely the same. We start to prove the vanishing of $\mathrm{Ext}^l(\pi^*\mathcal{T}_X\otimes \mathcal{V}_i,\pi^*\mathcal{T}_X\otimes \mathcal{V}_j)$ for $l>0$. By adjunction of $\pi^*$ and $\pi_*$ and the projection formula we obtain
\begin{eqnarray*}
\mathrm{Hom}(\pi^*\mathcal{T}_X\otimes \mathcal{V}_i,\pi^*\mathcal{T}_X\otimes \mathcal{V}_j[l])\simeq \mathrm{Hom}(\mathcal{T}_X,\mathcal{T}_X\otimes \mathbb{R}\pi_*(\mathcal{V}_i^{\vee}\otimes \mathcal{V}_j)[l]). 
\end{eqnarray*} Lemma 4.10 together with the assumption yields
\begin{center}
$\mathrm{Ext}^l(\pi^*\mathcal{T}_X\otimes \mathcal{V}_i,\pi^*\mathcal{T}_X\otimes \mathcal{V}_j)=0$ for $l>0$.
\end{center} Note that for $i=j$ we have with Lemma 4.10 
\begin{center}
$\mathrm{Hom}(\mathcal{T}_X,\mathcal{T}_X\otimes \mathbb{R}\pi_*(\mathcal{V}_i^{\vee}\otimes \mathcal{V}_i)[l])\simeq \mathrm{Ext}^l(\mathcal{T}_X,\mathcal{T}_X)=0$ for $l\neq 0$, 
\end{center} since $\mathcal{T}_X$ is a tilting bundle on $X$ by assumption. The generating property of $\mathcal{T}=\bigoplus_i\pi^*\mathcal{T}_X\otimes \mathcal{V}_i$ follows from Lemma 4.11, since $\mathcal{T}_X$ generates $D^b(X)$ by assumption. Since $\mathcal{Q}$ is by assumption smooth over $\mathbb{C}$, Theorem 2.3 implies that $\mathrm{End}(\mathcal{T})$ has finite global dimension. This completes the proof.
\end{proof}

We see that an obstruction for $\mathcal{T}=\bigoplus_i\pi^*\mathcal{T}_X\otimes \mathcal{V}_i$ to be a tilting object on the quadric bundle is that one has to verify $\mathrm{Ext}^l(\mathcal{T}_X,\mathcal{T}_X\otimes \pi_*(\mathcal{V}_j\otimes \mathcal{V}^{\vee}_i))=0$ for $l\neq 0$. For the moment, we continue by considering a more general situation as above (Proposition 4.12), namely considering flat fibrations. This was done by Costa, Di Rocco and Mir\'o--Roig \cite{CRMR} for fibrations with typical fiber $F$ and by Samokhin \cite{SAM} for arbitrary fibrations. Note that the results in loc.cit. concerned full (strongly) exceptional collections often consisting of invertible sheaves. Below we want to generalize the result of Costa, Di Rocco and Mir\'o--Roig \cite{CRMR}, Theorem 2.8 by considering arbitrary fibrations and allowing the base scheme to admit a tilting bundle whose direct summands are not necessarily invertible sheaves. This generalization also shows that the cohomologically assumption of Proposition 4.12 can be manged to be fulfilled. We begin with a lemma.

\begin{lem}
Let $\pi:X\rightarrow Z$ be a flat proper morphism between two smooth projective schemes over an algebraically closed field of characteristic zero. Let $\mathcal{E}_1,...,\mathcal{E}_n$ be a set of locally free sheaves in $D^b(X)$ and suppose that for any point $z\in Z$ the restriction $\mathcal{E}^z_i=\mathcal{E}_i\otimes \mathcal{O}_{X_z}$ to the fiber $X_z$ is a full strongly exceptional collection for $D^b(X_z)$. Then the following holds:
\[\mathbb{R}^s\pi_*(\mathcal{E}_q\otimes\mathcal{E}^{\vee}_p)=\begin{cases}
0& \text{for } s>0\\
0& \text{for s = 0 and } q<p\\
\pi_*(\mathcal{E}_q\otimes\mathcal{E}^{\vee}_p) &\text{for s = 0 and } q\geq p
\end{cases}\]
\end{lem}
\begin{proof}
Consider the diagramm
\begin{displaymath}
\begin{xy}
  \xymatrix{
      X_z\ar[d]_{\pi_z}\ar[r]^{\tilde{i}_z}    &   X\ar[d]^{\pi}                   \\
      z\ar[r]^{i_z}             &   Z             
  }
\end{xy}
\end{displaymath} and note that $\pi$ is and flat and proper. Therefore flat base change holds and we conclude
\begin{center}
$i_z^* \mathbb{R}^{\bullet}\pi_{*}(\mathcal{E}_q\otimes \mathcal{E}^{\vee}_p)=\mathbb{H}^{\bullet}(X_z,\mathcal{E}_q\otimes \mathcal{E}^{\vee}_p\otimes \mathcal{O}_{X_z})$.
\end{center}
The rest of the proof follows exactly the lines of the proof of the claim in \cite{CRMR}, p.10006 (See also \cite{SAM}, p.5 and p.6.). 
\end{proof}
To prove Theorem 4.15 below we need a further observation. It is the following fact, essentially proved in \cite{SAM}, Theorem 3.1.
\begin{lem}
Let $\pi:X\rightarrow Z$ be as above and $\mathcal{E}_i$ the sheaves from Lemma 4.13. Suppose that $D^b(Z)$ is generated by some object $\mathcal{A}$, then $D^b(X)$ is generated by the object $\mathcal{R}=\bigoplus^n_{i=1}\pi^*(\mathcal{A})\otimes \mathcal{E}_i$.
\end{lem}
\begin{proof}
Samokhin \cite{SAM}, Theorem 3.1 proved that the functor $\pi^*(-)\otimes \mathcal{E}_i:D^b(Z)\rightarrow D^b(X)$ is fully faithful and that $D^b(X)=\langle \pi^*D^b(Z)\otimes \mathcal{E}_1,...,\pi^*D^b(Z)\otimes \mathcal{E}_n\rangle$ is a semiorthogonal decomposition. The full subcategories $\pi^*D^b(Z)\otimes\mathcal{E}_i$ consist of objects of the form $\pi^*\mathcal{M}\otimes \mathcal{E}_i$, where $\mathcal{M}\in D^b(Z)$. Therefore, the functor $\pi^*(-)\otimes \mathcal{E}_i$ from above induces an equivalence between $D^b(Z)$ and $\pi^*D^b(Z)\otimes\mathcal{E}_i$. Since $\mathcal{A}$ generates $D^b(Z)$, the object $\pi^*(\mathcal{A})\otimes \mathcal{E}_i$ generates $\pi^*D^b(Z)\otimes \mathcal{E}_i$ and hence $\mathcal{R}=\bigoplus^n_{i=1}\pi^*(\mathcal{A})\otimes\mathcal{E}_i$ generates $D^b(X)$.
\end{proof}

\begin{thm}
Let $\pi:X\rightarrow Z$ and $\mathcal{E}_i$ be as in Lemma 4.13 and suppose that $D^b(Z)$ admits a tilting bundle $\mathcal{T}$. Then there exists an ample invertible sheaf $\mathcal{M}$ on $Z$ such that $\mathcal{R}=\bigoplus^n_{i=1}\pi^*(\mathcal{T}\otimes \mathcal{M}^{\otimes i})\otimes \mathcal{E}_i$ is a tilting bundle for $D^b(X)$.
\end{thm}
\begin{proof}
We will show that there is an ample invertible sheaf $\mathcal{M}$ on $Z$ such that $\mathcal{R}=\bigoplus^n_{i=1}\pi^*(\mathcal{T}\otimes \mathcal{M}^{\otimes i})\otimes \mathcal{E}_i$ is a tilting object for $D^b(X)$. For the vanishing of Ext we therefore have to find the ample invertible sheaf $\mathcal{M}$ such that
\begin{center}
$\mathrm{Ext}^l(\pi^*(\mathcal{T}\otimes\mathcal{M}^{\otimes i})\otimes\mathcal{E}_i,\pi^*(\mathcal{T}\otimes\mathcal{M}^{\otimes j})\otimes \mathcal{E}_j)=0$, for $l>0$.
\end{center}
But this is equivalent to
\begin{center}
$H^l(X,\pi^*(\mathcal{T}\otimes\mathcal{T}^{\vee}\otimes\mathcal{M}^{\otimes (j-i)})\otimes\mathcal{E}_j\otimes\mathcal{E}^{\vee}_i)=0$, for $l>0$.
\end{center}
Applying the Leray spectral sequence for the morphism $\pi$ (see \cite{HUY}, p.74), one gets
\begin{eqnarray*}
H^r(Z,\mathbb{R}^s\pi_*(\pi^*(\mathcal{T}\otimes\mathcal{T}^{\vee}\otimes\mathcal{M}^{\otimes (j-i)})\otimes\mathcal{E}_j\otimes\mathcal{E}^{\vee}_i))\Longrightarrow\\
H^{r+s}(X,\pi^*(\mathcal{T}\otimes\mathcal{T}^{\vee}\otimes\mathcal{M}^{\otimes (j-i)})\otimes\mathcal{E}_j\otimes\mathcal{E}^{\vee}_i).
\end{eqnarray*}
With the projection formula we find
\begin{center}
$\mathbb{R}^s\pi_*(\pi^*(\mathcal{T}\otimes\mathcal{T}^{\vee}\otimes\mathcal{M}^{\otimes (j-i)})\otimes\mathcal{E}_j\otimes\mathcal{E}^{\vee}_i)\simeq \mathcal{T}\otimes\mathcal{T}^{\vee}\otimes\mathcal{M}^{\otimes (j-i)}\otimes\mathbb{R}^s\pi_*(\mathcal{E}_j\otimes\mathcal{E}^{\vee}_i)$.
\end{center}
Now from Lemma 4.13 we know that $\mathbb{R}^s\pi_*(\mathcal{E}_j\otimes\mathcal{E}^{\vee}_i)$ is non-vanishing only for $s=0$ and $j\geq i$ and that in this case one has $\mathbb{R}^s\pi_*(\mathcal{E}_j\otimes\mathcal{E}^{\vee}_i)\simeq \pi_*(\mathcal{E}_j\otimes\mathcal{E}^{\vee}_i)$. Thus for $j<i$ we have $\mathbb{R}^s\pi_*(\mathcal{E}_j\otimes\mathcal{E}^{\vee}_i)=0$ and therefore 
\begin{center}
$H^r(Z,\mathcal{T}\otimes\mathcal{T}^{\vee}\otimes\mathcal{M}^{\otimes (j-i)}\otimes\mathbb{R}^s\pi_*(\mathcal{E}_j\otimes\mathcal{E}^{\vee}_i))=0$. 
\end{center}
Therefore we find 
\begin{center}
$H^{l}(X,\pi^*(\mathcal{T}\otimes\mathcal{T}^{\vee}\otimes\mathcal{M}^{\otimes (j-i)})\otimes\mathcal{E}_j\otimes\mathcal{E}^{\vee}_i)=0$, 
\end{center}
for $l>0$ by above spectral sequence. It remains the case $j\geq i$. For $j=i$ we have $\mathbb{R}^s\pi_*(\mathcal{E}_i\otimes\mathcal{E}^{\vee}_i)\simeq \pi_*(\mathcal{E}_i\otimes\mathcal{E}^{\vee}_i)\simeq \mathcal{O}_Z$ (see \cite{SAM}, p.5 right after (3.10)). From this we get 
\begin{eqnarray*}
H^r(Z,\mathcal{T}\otimes\mathcal{T}^{\vee}\otimes\mathcal{M}^{\otimes (i-i)}\otimes\mathbb{R}^s\pi_*(\mathcal{E}_i\otimes\mathcal{E}^{\vee}_i))&\simeq& H^r(Z,\mathcal{T}\otimes \mathcal{T}^{\vee}\otimes \mathcal{O}_Z)\\
&\simeq&\mathrm{Ext}^r(\mathcal{T},\mathcal{T})=0, 
\end{eqnarray*}
for $r>0$, since $\mathcal{T}$ is a tilting bundle for $D^b(X)$ by assumption. Again by the above spectral sequence we conclude 
\begin{center}
$H^{l}(X,\pi^*(\mathcal{T}\otimes\mathcal{T}^{\vee}\otimes\mathcal{M}^{\otimes (i-i)})\otimes\mathcal{E}_i\otimes\mathcal{E}^{\vee}_i)=0$,
\end{center} for $l>0$. Finally, it remains the case $j>i$. For this, we again consider the above spectral sequence and see that it becomes
\begin{eqnarray*}
H^r(Z, \mathcal{T}\otimes \mathcal{T}^{\vee}\otimes \mathcal{M}^{\otimes (j-i)}\otimes \pi_*(\mathcal{E}_j\otimes\mathcal{E}^{\vee}_i))\Longrightarrow\\
H^{r}(X,\pi^*(\mathcal{T}\otimes\mathcal{T}^{\vee}\otimes\mathcal{M}^{\otimes (j-i)})\otimes\mathcal{E}_j\otimes\mathcal{E}^{\vee}_i). 
\end{eqnarray*}
Since there are only finitely many $\mathcal{E}_i$ and $Z$ is projective, we can choose an ample invertible sheaf $\mathcal{N}$ on $Z$ and an integer $m>>0$ such that for $\mathcal{M}=\mathcal{N}^{\otimes m}$ we have 
\begin{center}
$H^r(Z,\mathcal{T}\otimes\mathcal{T}^{\vee}\otimes\mathcal{M}^{\otimes (j-i)}\otimes\pi_*(\mathcal{E}_j\otimes\mathcal{E}^{\vee}_i))=0$ for $r>0$.
\end{center}
This finally yields 
\begin{center}
$H^l(X,\pi^*(\mathcal{T}\otimes\mathcal{T}^{\vee}\otimes\mathcal{M}^{\otimes (j-i)})\otimes\mathcal{E}_j\otimes\mathcal{E}^{\vee}_i)=0$ for $l>0$
\end{center} and therefore
\begin{center}
$\mathrm{Ext}^l(\pi^*(\mathcal{T}\otimes\mathcal{M}^{\otimes i})\otimes\mathcal{E}_i,\pi^*(\mathcal{T}\otimes\mathcal{M}^{\otimes j})\otimes \mathcal{E}_j)=0$ for $l>0$.
\end{center} The generating property of $\mathcal{R}=\bigoplus^n_{i=1}\pi^*(\mathcal{T}\otimes \mathcal{M}^{\otimes i})\otimes \mathcal{E}_i$ follows from Proposition 2.5 and Lemma 4.15 as $\mathcal{T}$ generates $D^b(Z)$ by assumption. Note that the global dimension of $\mathrm{End}(\mathcal{R})$ is finite, since $X$ is smooth over $k$. Theorem 2.3 completes the proof. 
\end{proof}
We immediately have the following consequence:
\begin{thm}
Let $X$, $\mathcal{E}$ and $\mathcal{Q}$ be as in Proposition 4.12. Suppose $H^1(X,\mathbb{Z}/2\mathbb{Z})=0$ and that $\mathcal{E}$ is orthogonal and carries a spin structure. Suppose furthermore that $X$ admits a tilting bundle. Then $\mathcal{Q}\subset \mathbb{P}(\mathcal{E})$ admits a tilting bundle too.
\end{thm}
\begin{proof}
We let $\mathcal{E}_i$ be the locally free sheaves of collection (3) respectively (4) from above. From Lemma 4.10 we conclude that these $\mathcal{E}_i$ satisfy the condition of Theorem 4.15. Thus $D^b(\mathcal{Q})$ admits a tilting bundle.   
\end{proof}
\begin{rema}
The assumption $H^1(X,\mathbb{Z}/2\mathbb{Z})=0$ of the above theorem is fulfilled for instance if $X$ is supposed to be simply connected. 
\end{rema}
\begin{exam}
Let $X$ be a smooth projective quadric over $\mathbb{C}$ and $\mathcal{E}$ a finite rank orthogonal locally free sheaf carrying a spin structure. Then for any symmetric quadratic form $q\in \Gamma(X,\mathrm{Sym}^2(\mathcal{E}^{\vee}))$ which is non-degenerate on each fiber, the associated quadric bundle $\mathcal{Q}\rightarrow X$ admits a tilting bundle. 
\end{exam}
\begin{exam}
Theorem 4.15 also gives a proof for Theorem 4.6. Note that we conclude from Lemma 4.5 that the locally free sheaves $\mathcal{E}_{\lambda}=\Sigma^\lambda(\mathcal{R})$ on $\mathrm{Grass}(l,\mathcal{E})$ give rise to a full strongly exceptional collection on each fiber and therefore Theorem 4.15 implies that $\mathrm{Grass}(l,\mathcal{E})\rightarrow X$ admits a tilting object if $X$ admits a tilting bundle. Note that this result does not follow from the result of Costa, Di Rocco and Mir\'o-Roig \cite{CRMR} since the sheaves $\Sigma^\lambda(\mathcal{R})$ are not invertible.
\end{exam}

\section{Tilting objects on twisted forms of some relative flag varieties}
In this section we study twisted forms of certain relative flags and prove that for some of them their bounded derived category of coherent sheaves admit a tilting object.\\

We first consider classical homogeneous varieties over an algebraically closed field of characteristic zero. Recall that the semisimple algebraic groups over an algebraically closed field $k$ of characteristic zero are classified by Dynkin diagrams that fall into types $A$, $B$, $C$, $D$, $E$ and $F$ (see \cite{NBOU}). The classical semisimple algebraic groups are given by $\mathrm{SL}_k(n+1)$, $\mathrm{SO}_k(2n)$, $\mathrm{SO}_k(2n+1)$ and $\mathrm{Sp}_k(2n)$ and the corresponding Dynkin diagrams are $A_n$, $B_n$, $C_n$ and $D_n$ (see \cite {FU0}). Samokhin \cite{SAM} proved that for $G$ a semisimple algebraic group of classical type and $B$ a Borel subgroup the flag variety $G/B$ admits a full exceptional collection. As a consequence of the results obtained in the previews section we get that $G/B$ also admits a tilting object.  
\begin{thm}
Let $G$ be a semisimple algebraic group of classical type, $B$ a Borel subgroup and $G/B$ the flag variety of $G$. Then $D^b(G/B)$ admits a tilting object.
\end{thm}
\begin{proof}
The strategy of the proof is simply by considering every possible case and by applying parabolic induction. First note that for a semisimple group $G$ and a Borel subgroup $B$ the homogeneous variety $G/B$ is given as
\begin{center}
$G/B=G_1/B_1\times...\times G_r/B_r$
\end{center} where $G_i$ are simple groups and $B_i\subset G_i$ Borel subgroups. In view of Proposition 2.6 we can restrict ourselves to the case that $G$ is simple. The homogeneous varieties of the groups of type $A_n$ were treated by Kapranov \cite{KAP2}, Theorem 3.10. These homogeneous varieties admit full strongly exceptional collections consisting of locally free sheaves and hence tilting bundles. Note that this case also follows from Corollary 4.7 of the present work. We now consider the flag varieties of type $C_n$. These correspond to the group $\mathrm{Sp}_k(2n)$ and are partial isotropic flags in a symplectic vector space $V$. Since the complete isotropic flag $\mathrm{Sp}_k(2n)/B$ can be obtained as an iteration of projective bundles over $\mathbb{P}^{2n-1}$ (see \cite{SAM}, p.7), applying Theorem 4.1 several times yields that these homogeneous varieties have a tilting object. We proceed with the flag varieties of type $B_n$ and $D_n$. We restrict ourselves to the case of the orthogonal group corresponding to the Dynkin diagram $B_n$, the case of $D_n$ being similar. The arguments of the proof of Theorem 4.1 in \cite{SAM} show that the complete flags corresponding to $B_n$ are obtained as a successive iteration of smooth quadric fibrations over a smooth quadric $Q_{2n-1}\subset \mathbb{P}^{2n}$. To be precise, the complete flag $\mathrm{SO}_k(2n+1)/B$ is equipped with universal bundles $\mathcal{W}_i$, where $i=1,...,n$, fitting into a sequence
\begin{center}
$0\subset \mathcal{W}_1\subset...\subset\mathcal{W}_n\subset\mathcal{W}^{\perp}_n\subset...\subset\mathcal{W}^{\perp}_1\subset V\otimes \mathcal{O}_{\mathrm{SO}_k(2n+1)/B}$.
\end{center} Remind that $V$ is the $2n+1$-dimensional vector space equipped with a non-degenerate symmetric form $q\in S^2V^*$. Here $\mathcal{W}^{\perp}_i$ is the locally free sheaf orthogonal to $\mathcal{W}_i$ with respect to $q$. The flag $\mathrm{SO}_k(2n+1)/B$ is now the iteration of quadric fibrations $\mathcal{Q}_i\subset \mathbb{P}_{X_{i-1}}(\mathcal{M}_{i-1})$, where $\mathcal{M}_{i-1}=\mathcal{W}^{\perp}_{i-1}/\mathcal{W}_{i-1}$ and $X_{i-1}$ is the quadric fibration $\mathcal{Q}_{i-1}\subset \mathbb{P}_{X_{i-2}}(\mathcal{M}_{i-2})$. The base of this iteration is, as mentioned above, the smooth quadric $Q_{2n-1}$. In this special situation, at every step of the iteration the relative spinor bundles exist (see \cite{SAM}, p.9). Since the smooth quadric $Q_{2n-1}\subset \mathbb{P}^{2n}$ admits a full strongly exceptional collection of locally free sheaves, and hence a tilting bundle, (see \cite{KAP1}), applying Theorem 4.15 several times provides us with a tilting bundle for $\mathrm{SO}_k(2n+1)/B$. This completes the proof. 
\end{proof}
\begin{rema}
Note that for an arbitrary parabolic subgroup $P\subset G$ the homogeneous variety $G/P$ can be obtained as an iteration of fibrations with fibers being of the form $G_i/P_i$, where $G_i$ are semisimple and $P_i\subset G_i$ are maximal parabolic subgroups. Moreover, if $G$ is of type $B$, $C$ or $D$ then all $G_i$ are also of type $B$, $C$ or $D$. For $G$ a simply connected simple group of type $B$, $C$ or $D$ Kuznetsov and Polishchuk \cite{KUZ2} constructed exceptional collections on $G/P$, where $P$ is a maximal parabolic subgroup corresponding to vertices of the Dynkin diagramm. The $G$-equivariant structure of these collections allows to construct relative collections on any fibration with fiber $G/P$. If one manages to prove that this collection is full and strong, applying Theorem 4.15 several times would give us a tilting bundle on any $G/P$, with $P$ being an arbitrary parabolic subgroup. 
\end{rema}
The above theorem shows that at least for $G$ a semisimple algebraic group of classical type and $B$ a Borel subgroup, the variety $G/B$ admits a tilting bundle. As mentioned earlier, it is conjectured that $G/P$, where $P$ is an arbitrary parabolic subgroup, admits a full (strongly) exceptional collection. Up to now only partial results in favor of this conjecture were obtained. For details we refer to \cite{KUZ2} and references therein. To provide further evidence for Conjecture 3.5 it is natural to start with the investigation if the twisted forms of the homogeneous varieties $G/B$ of Theorem 5.1 admit tilting bundles. For this, we study the more general situation of twisted forms of the relative flag varieties considered in Section 4.\\

We roughly recall the basics of generalized Brauer--Severi schemes (see \cite{LSW}). Let $X$ be a noetherian $k$-scheme and $\mathcal{A}$ a sheaf of Azumaya algebras of rank $n^2$ over $X$ (see \cite{GRO}, \cite{GRO1} for details on Azumaya algebras). For an integer $1\leq l<n$ the generalized Brauer--Severi scheme $p:\mathrm{BS}(l,\mathcal{A})\rightarrow X$ is defined as the scheme representing the functor $F:\mathrm{Sch}/X\rightarrow \mathrm{Sets}$, where $(\psi:Y\rightarrow X )$ is mapped to the set of left ideals $\mathcal{J}$ of $\psi^*\mathcal{A}$ such that $\psi^*\mathcal{A}/\mathcal{J}$ is locally free of rank $n(n-l)$. By definition, there is an \'etale covering $U\rightarrow X$ and a locally free sheaf $\mathcal{E}$ of rank $n$ with the following trivializing diagram:
\begin{displaymath}
\begin{xy}
  \xymatrix{
      \mathrm{Grass}(l,\mathcal{E}) \ar[r]^{\pi} \ar[d]_{q}    &   \mathrm{BS}(l,\mathcal{A}) \ar[d]^{p}                   \\
      U \ar[r]^{g}             &   X             
  }
\end{xy}
\end{displaymath}
In the same way one defines the twisted relative flag $\mathrm{BS}(l_1,...,l_m,\mathcal{A})$ as the scheme representing the functor $F:\mathrm{Sch}/X\rightarrow \mathrm{Sets}$, where $(\psi:Y\rightarrow X )$ is mapped to the set of left ideals $\mathcal{J}_1\subset...\subset \mathcal{J}_m$ of $\psi^*\mathcal{A}$ such that $\psi^*\mathcal{A}/\mathcal{J}_i$ is locally free of rank $n(n-l_i)$. As for the generalized Brauer--Severi schemes, there is an \'etale covering $U\rightarrow X$ and a locally free sheaf $\mathcal{E}$ of rank $n$ with diagram
\begin{displaymath}
\begin{xy}
  \xymatrix{
      \mathrm{Flag}_U(l_1,...,l_m,\mathcal{E}) \ar[r]^{\pi} \ar[d]_{q}    &   \mathrm{BS}(l_1,...,l_m,\mathcal{A}) \ar[d]^{p}                   \\
      U \ar[r]^{g}             &   X             
  }
\end{xy}
\end{displaymath}
Note that the usual Brauer--Severi schemes are obtained from the generalized one by setting $l=1$. In this case one has a well known one-to-one correspondence between sheaves of Azumaya algebras of rank $n^2$ on $X$ and Brauer--Severi schemes of relative dimension $n-1$ via $\check{H}^1(X_{et}, \mathrm{PGL}_n)$ (see \cite{GRO}). Note that if the base scheme $X$ is a point a sheaf of Azumaya algebras on $X$ is a central simple $k$-algebra and the generalized Brauer--Severi schemes are the generalized Brauer--Severi varieties considered in Section 3. 
Now let $X$ be a smooth projective and integral $k$-scheme becoming rational after some finite separable extension. For a sheaf of Azumaya algebras $\mathcal{A}$ on $X$ let $p:\mathrm{BS}(l,\mathcal{A})\rightarrow X$ be the generalized Brauer--Severi scheme. If a scheme $Y$ is a rational smooth projective and integral $K$-scheme, the Brauer group of $\mathrm{Br}(Y)$ equals $\mathrm{Br}(K)$ (see \cite{YY}, Theorem 1.2.28). In our situation this implies that there exists a finite Galois extension $k\subset L$ such that we have the following diagram 
\begin{displaymath}
\begin{xy}
  \xymatrix{
      \mathrm{Grass}(l,\mathcal{E}) \ar[r]^{\tilde{\pi}} \ar[d]_{q}    &   \mathrm{BS}(l,\mathcal{A}) \ar[d]^{p}                   \\
      X_L \ar[r]^{\pi}             &   X             
  }
\end{xy}
\end{displaymath}
where $\mathcal{E}$ is a locally free sheaf on $X_L$ such that $\pi^*\mathcal{A}\simeq \mathcal{E}nd(\mathcal{E})$. Note that $\mathrm{Grass}(l,\mathcal{E}nd(\mathcal{E}))$ is naturally isomorphis to $\mathrm{Grass}(l,\mathcal{E})$. Denote by $G=\mathrm{Gal}(L|k)$ the Galois group and by $p_g:X_L\rightarrow X_L$ and $p_g:\mathrm{Grass}(l,\mathcal{E})\rightarrow \mathrm{Grass}(l,\mathcal{E})$ the morphisms induced by $g^{-1}:L\rightarrow L$. Applying fibrational techniques and descent theory, also used in the present work, Yan \cite{YY}, Theorem 4.1.16 proved the following result:
\begin{thm}
Let $X$ be a smooth projective $k$-scheme and $p:\mathrm{BS}(1,\mathcal{A})\rightarrow X$ the Brauer--Severi scheme. Suppose that $X$ becomes rational after a finite Galois extension $k\subset L$ and that $\mathrm{Gal}(L|k)$ acts trivially on $\mathrm{Pic}(X_L)$. If $X$ admits a tilting bundle, then $\mathrm{BS}(1,\mathcal{A})$ admits a tilting bundle too.
\end{thm}
In hole generality Bernardara \cite{BER} proved that Brauer--Severi schemes always have a semiorthogonal decomposition. Exploiting Theorem 4.6 we can now generalize Theorem 3.3 respectively 5.3, at least in characteristic zero, and obtain:
\begin{thm}
Let $k$ be a field of characteristic zero, $X$ a smooth projective $k$-scheme and $p:\mathrm{BS}(l,\mathcal{A})\rightarrow X$ the generalized Brauer--Severi scheme. Suppose that $X$ becomes rational after a separable field extension $k\subset L$. If $X$ admits a tilting bundle, then $\mathrm{BS}(d,\mathcal{A})$ admits a tilting bundle too.
\end{thm}
\begin{proof}
Since $X_L$ is rational, $\mathrm{Br}(X_L)=\mathrm{Br}(L)$ (see \cite{YY}, Theorem 1.2.28). We therefore have the following diagram

\begin{displaymath}
\begin{xy}
  \xymatrix{
      \mathrm{Grass}(l,\mathcal{E}) \ar[r]^{\tilde{\pi}} \ar[d]_{q}    &   \mathrm{BS}(l,\mathcal{A}) \ar[d]^{p}                   \\
      X_{\bar{k}} \ar[r]^{\pi}             &   X             
  }
\end{xy}
\end{displaymath} and hence $\mathrm{BS}(l,\mathcal{A})\otimes_k \bar{k}\simeq \mathrm{Grass}(l,\mathcal{E})$. By Proposition 2.8 it suffices to prove the existence of a tilting object on $\mathrm{Grass}(l,\mathcal{E})$ which descents to an object on $\mathrm{BS}(l,\mathcal{A})$. Denote by $\mathcal{R}$ the tautological subbundle on $\mathrm{Grass}(l,\mathcal{E})$ and by $\lambda$ a partition with at most $l$ rows and at most $n-l$ columns. Let $\mathcal{L}$ be some invertible sheaf on $X$ and $\mathcal{L}'=\mathcal{L}\otimes_k\bar{k}$. Note that $\Sigma^{\lambda}(\mathcal{R})^{\oplus n_{\lambda}}$ descents to a sheaf on $\mathrm{BS}(l,\mathcal{A})$ for suitable $n_{\lambda}$. Now let 
\begin{eqnarray*}
\mathcal{J}_{\lambda}=(\Sigma^{\lambda}(\mathcal{R})^{\oplus n_{\lambda}})\otimes q^*\mathcal{L}'^{\otimes |\lambda|})
\end{eqnarray*} where $\mathcal{L}'$ is invertible on $X_{\bar{k}}$. By construction the locally free sheaf $\mathcal{J}_{\lambda}$ descents to a sheaf on $\mathrm{BS}(l,\mathcal{A})$. Notice that $q^*\mathcal{L}'$ descents since the above diagram is commutative. Let $P(l,n-l)$ be the set of partitions $\lambda$ with at most $l$ rows and at most $n-l$ columns and choose a total order $\prec$ on this set where $\lambda \prec \mu$ if $|\lambda|<|\mu|$. Now let $P'$ be the set $P(l,n-l)$ equipped with this total order and $\mathcal{T}$ the tilting bundle on $X$. Consider the locally free sheaf
\begin{eqnarray*}
\mathcal{S}=\bigoplus_{\lambda\in P'}q^*\pi^*\mathcal{T}\otimes \mathcal{J}_{\lambda}
\end{eqnarray*} on $\mathrm{Grass}(l,\mathcal{E})$. Note that by construction $\mathcal{S}$ descents to a sheaf on $\mathrm{BS}(l,\mathcal{A})$. We claim that $\mathcal{S}$ is a tilting sheaf on $\mathrm{Grass}(l,\mathcal{E})$ for a suitable $\mathcal{L}$. By Proposition 2.7 the locally free sheaf $\pi^*\mathcal{T}$ is a tilting bundle on $X_{\bar{k}}$. In the following we denote this sheaf simply by $\mathcal{T}'$. With this notation we have
\begin{eqnarray*}
\mathcal{S}&= &\bigoplus_{\lambda\in P'}q^*\mathcal{T}'\otimes \mathcal{J}_{\lambda}\\
 &= &\bigoplus_{\lambda\in P'}q^*(\mathcal{T}'\otimes \mathcal{L}'^{\otimes|\lambda|})\otimes (\Sigma^{\lambda}(\mathcal{R})^{\oplus n_{\lambda}})
\end{eqnarray*}
Since $X$ is projective we choose a very ample sheaf $\mathcal{M}$ on $X$ and consider $\mathcal{M}^{\otimes n}$ for $n>0$. Note that since $\mathcal{M}$ is very ample on $X$, $\mathcal{M}\otimes_k \bar{k}$ is very ample, and therefore ample, on $X_{\bar{k}}$. If we set $\mathcal{L}=\mathcal{M}^{\otimes n}$ and $\mathcal{L}'= (\mathcal{M}^{\otimes n})\otimes_k \bar{k}=(\mathcal{M}\otimes_k \bar{k})^{\otimes n}$, Proposition 2.4, Example 4.19 and the proof of Theorem 4.15 show that we can choose $n>>0$ in such a way that $\mathcal{S}$ becomes a tilting sheaf on $\mathrm{Grass}(l,\mathcal{E})$. Since $\mathcal{S}$ descents, Proposition 2.8 provides us with a tilting sheaf on $\mathrm{BS}(l,\mathcal{A})$. This completes the proof.
\end{proof}
In characteristic zero Baek \cite{BAE} constructed  semiorthogonal decompositions for generalized Brauer--Severi schemes and twisted relative flag varieties in hole generality. But it is not yet clear how to find in this generality tilting bundles on generalized Brauer--Severi schemes, provided the base scheme admits one. Also notice that the proof of Theorem 5.4 shows that Theorem 5.3 holds true without the assumption that the Galois group has to act trivially on the Picard group. A direct consequence of Theorem 5.4 is the following:
\begin{cor}
Let $X$ be as above and $p:\mathrm{BS}(l_1,...,l_m,\mathcal{A})\rightarrow X$ a twisted relative flag variety. Suppose that $X$ becomes rational after a separable field extension. If $X$ admits a tilting bundle, then $\mathrm{BS}(l_1,...,l_m,\mathcal{A})$ admits a tilting bundle too.
\end{cor}
\begin{proof}
We shall prove by induction on $m$. The case $m=1$ follows from Theorem 5.4. Now assume that the result holds for $m-1$. We have projections
\begin{center}
$\mathrm{Flag}_{X_{\bar{k}}}(l_1,...,l_m,\mathcal{E})\stackrel{q_m}\longrightarrow...\stackrel{q_2}\longrightarrow \mathrm{Flag}_{X_{\bar{k}}}(l_m,\mathcal{E})\stackrel{q_1}\longrightarrow X_{\bar{k}}$
\end{center} and 
\begin{center}
$\mathrm{BS}(l_1,...,l_m,\mathcal{A})\stackrel{p_m}\longrightarrow...\stackrel{p_2}\longrightarrow \mathrm{BS}(l_m,\mathcal{A})\stackrel{p_1}\longrightarrow X$.
\end{center} Now let $\mathcal{R}\subset (q_1\circ...\circ q_m)^*\mathcal{E}$ be the tautological sheaf on $\mathrm{Flag}_{X_{\bar{k}}}(l_2,...,l_m,\mathcal{E})$ and let $\mathcal{A}'$ be the sheaf of Azumaya algebras on $\mathrm{BS}(l_2,...,l_m,\mathcal{A})$ obtained from $\mathcal{E}nd(\mathcal{R})$ by descent. This implies $\mathrm{Flag}_{X_{\bar{k}}}(l_1,...,l_m,\mathcal{E})=\mathrm{Grass}_{\mathrm{Flag}_{X_{\bar{k}}}(l_2,...,l_m,\mathcal{E})}(l_1,\mathcal{R})$ and $\mathrm{BS}(l_1,...,l_m,\mathcal{A})=\mathrm{BS}(l_1,\mathcal{A}')$. 
The assertion then follows from Theorem 5.4. 
\end{proof}
As mentioned in the introduction, we believe that adopting the approach developed by Buchweitz, Leuschke and Van den Bergh \cite{BLB} in the relative setting should give us a tilting bundle on $\mathrm{BS}(l,\mathcal{A})$ without the assumption on $k$ being of characteristic zero.

In the particular case when $X=\mathrm{Spec}(k)$ is a point, the Galois group $\mathrm{Gal}(L|k)$ clearly acts trivially on $\mathrm{Pic}(\mathrm{Spec}(L))$. We now explain how Theorem 5.3 and Corollary 5.5 can be applied to get tilting objects on twisted forms of some homogeneous varieties, providing further evidence for Conjecture 3.5. We start with twisted forms of homogeneous varieties of type $C_n$. Let $\mathrm{Sp}(2n)/B$ be the complete isotropic flag over an algebraically closed field of characteristic zero, mentioned in Theorem 5.1. This flag variety can be obtained as an iteration of projective bundles starting from $\mathbb{P}^{2n-1}$, lets say
\begin{center}
$\mathbb{P}(\mathcal{E}_m)\stackrel{q_m}\longrightarrow...\stackrel{q_2}\longrightarrow \mathbb{P}(\mathcal{E}_1)\stackrel{q_1}\longrightarrow \mathbb{P}^{2n-1}$.
\end{center}
Now let $X$ be a twisted from of $\mathbb{P}^{2n-1}$, that is a Brauer--Severi variety of dimension $2n-1$ over a field of characteristic zero, and $\mathcal{A}_1$ a sheaf of Azumaya algebras on $X$ obtained from $\mathcal{E}nd(\mathcal{E}_1)$ by descent. Inductively we get sheaves of Azumaya algebras $\mathcal{A}_i$ on $\mathrm{BS}(\mathcal{A}_{i-1})$ obtained from $\mathcal{E}nd(\mathcal{E}_i)$ by descent. Therefore we get a twisted form 
\begin{center}
$\mathrm{BS}(\mathcal{A}_m)\stackrel{p_m}\longrightarrow...\stackrel{p_2}\longrightarrow \mathrm{BS}(\mathcal{A}_1)\stackrel{q_1}\longrightarrow X$
\end{center} of the flag variety $\mathrm{Sp}(2n)/B$ that was previously given as 
\begin{center}
$\mathbb{P}(\mathcal{E}_m)\stackrel{q_m}\longrightarrow...\stackrel{q_2}\longrightarrow \mathbb{P}(\mathcal{E}_1)\stackrel{q_1}\longrightarrow \mathbb{P}^{2n-1}$.
\end{center}
Note that after base change to a finite Galois extension $k\subset L$ the twisted form
\begin{center}
$\mathrm{BS}(\mathcal{A}_m)\stackrel{p_m}\longrightarrow...\stackrel{p_2}\longrightarrow \mathrm{BS}(\mathcal{A}_1)\stackrel{q_1}\longrightarrow X$
\end{center} becomes 
\begin{center}
$\mathbb{P}(\widetilde{\mathcal{E}}_m)\stackrel{q_m}\longrightarrow...\stackrel{q_2}\longrightarrow \mathbb{P}(\widetilde{\mathcal{E}}_1)\stackrel{q_1}\longrightarrow \mathbb{P}^{2n-1}$
\end{center} for certain locally free sheaves $\widetilde{\mathcal{E}}_i$.
Hence applying Theorem 5.3 several times yields:
\begin{cor}
For a field $k$ of characteristic zero, let $X$ be the above twisted form of the homogeneous variety $\mathrm{Sp}(2n)/B$. Then $D^b(X)$ admits a tilting object.
\end{cor}
\begin{proof}
Again we prove by induction on $m$. The case $m=1$ follows from Theorem 5.3. Now assume the result holds for $m-1$. We have the projection $q_m:\mathbb{P}(\widetilde{\mathcal{E}}_m)\rightarrow \mathbb{P}(\widetilde{\mathcal{E}}_{m-1})$. Let $\mathcal{A}_m$ be the sheaf of Azumaya algebras on $\mathrm{BS}(\mathcal{A}_{m-1})$ obtained from $\mathcal{E}nd(\widetilde{\mathcal{E}}_m)$ by descent. Note that $\mathrm{BS}(\mathcal{A}_m)$ is a twisted form of $\mathbb{P}(\widetilde{\mathcal{E}}_m)$. Since $\mathrm{Gal}(L|k)$ acts trivially on $\mathbb{P}(\widetilde{\mathcal{E}}_{m-1})$, Theorem 5.3 yield the assertion.  
\end{proof}
For a field $k$ of characteristic zero, let $V$ be a $n^2$-dimensional $\bar{k}$-vector space and $\mathrm{Flag}(l_1,...,l_m,V)$ a partial flag variety. Consider a central simple $k$-algebra $A$ obtained from $\mathrm{End}(V)=M_n(\bar{k})$ by descent. Now let $X$ be a twisted form of the partial flag variety $\mathrm{Flag}(l_1,...,l_m,V)$, given as $\mathrm{BS}(l_1,...,l_m,A)$. Clearly we have $\mathrm{BS}(l_1,...,l_m,A)\otimes_k \bar{k}\simeq \mathrm{Flag}(l_1,...,l_m,V)$.
Corollary 5.5 immediately yields:
\begin{cor}
Let $X$ be a twisted form of the flag $\mathrm{Flag}(l_1,...,l_m,V)$, given as $\mathrm{BS}(l_1,...,l_m,A)$ from above, then $D^b(X)$ admits a tilting object.
\end{cor}
Note that the automorphism group of the Grassmannian $\mathrm{Grass}(d,n)$ is equal to $\mathrm{PGL}_n$ only if $2d\neq n$ (see \cite{BL}) and that therefore not all twisted forms of $\mathrm{Grass}(d,n)$ are related to central simple algebras, i.e., are not classified by $H^1(k,\mathrm{PGL}_n)$. To prove that twisted homogeneous varieties of type $B_n$ and $D_n$ admit tilting bundles, one can try to proceed as follows: Let $X$ be a smooth projective and integral $k$-scheme, becoming rational after a separable field extension $k\subset L$. Consider a sheaf of Azumaya algebras $\mathcal{A}$ on $X$ that has an involution $\sigma$ of the first kind. One then has an involution scheme $\mathrm{I}(\mathcal{A},\sigma)\rightarrow X$ which is a twisted form of a quadric bundle $\mathcal{Q}\rightarrow X_L$. Adapting the arguments given in \cite{BLU} in the relative setting should provide us with a tilting bundle on $\mathrm{I}(\mathcal{A},\sigma)\rightarrow X$. Once this can be proved to be true, the approach presented above should give tilting objects on twisted forms of homogeneous varieties of type $B_n$ and $D_n$. To work this out is the aim of a forthcoming paper of the author.

{\small MATHEMATISCHES INSTITUT, HEINRICH--HEINE--UNIVERSIT\"AT 40225 D\"USSELDORF, GERMANY}\\
E-mail adress: novakovic@math.uni-duesseldorf.de

\end{document}